\documentclass[11pt,a4paper,reqno]{amsart}

\usepackage{hyperref}
\usepackage[all]{xy}
\usepackage{amssymb}
\usepackage{amsmath}
\usepackage{appendix}

\usepackage{color}
\usepackage[usernames,dvipsnames]{xcolor}

\font\cyr=wncyr10 scaled \magstep1
\def\Sh{\text{\cyr Sh}}

\DeclareMathOperator{\cok}{coker}

\newcommand{\Nr}{\text{Nr}}

\newcommand{\fp}{\mathfrak{p}}

\newcommand{\fq}{\mathfrak{q}}

\newcommand{\Sp}{\text{Spec} \,}

\newcommand{\disc}{\text{disc}}

\DeclareMathOperator{\Res}{Res}

\chardef\bslash=`\\ 

\newtheorem{theorem}{Theorem}[section]
\newtheorem{thm}[theorem]{Theorem}
\newtheorem{prop}[theorem]{Proposition}
\newtheorem{lem}[theorem]{Lemma}
\newtheorem{cor}[theorem]{Corollary}

\numberwithin{equation}{section}

\newtheorem*{maintheorem*}{Main Theorem}
\newtheorem*{repeatprop}{Proposition}
\theoremstyle{definition}
\newtheorem{definition}[theorem]{Definition}
\newtheorem{remark}[theorem]{Remark}
\newtheorem{example}[theorem]{Example}

\newcommand{\surj}{\twoheadrightarrow}

\newcommand{\CO}{\mathcal{O}}

\newcommand{\Ob}{\un{\mathbf{O}}}
\newcommand{\OV}{\un{\mathbf{O}}_q}
\newcommand{\OVdm}{\un{\mathbf{O}}_{d,m}}
\newcommand{\Od}{\un{\mathbf{O}}_d}

\newcommand{\wt}{\widetilde}

\newcommand{\mk}{\medskip}
\renewcommand{\sectionmark}[1]{}
\renewcommand{\Im}{\operatorname{Im}}

\newcommand{\la}{\langle}
\newcommand{\ra}{\rangle}
\newcommand{\N}{\un{N}}

\newcommand{\diag}{\text{diag}}

\newcommand{\cl}{\mathfrak{cl}}

\newcommand{\ve}{\varepsilon}
\newcommand{\iy}{\infty}

\newcommand{\bk}{\bigskip}
\newcommand{\fc}{\frac}
\newcommand{\G}{\Gamma}
\newcommand{\g}{\gamma}

\newcommand{\Pic}{\mathrm{Pic~}}

\newcommand{\F}{\mathcal{F}}

\newcommand{\dl}{\delta}
\newcommand{\Dl}{\Delta}

\newcommand{\om}{\omega}
\newcommand{\Om}{\Omega}
\newcommand{\ov}{\overline}
\newcommand{\un}{\underline}

\newcommand{\BG}{\mathbb{G}}
\newcommand{\BF}{\mathbb{F}}

\newcommand{\BQ}{\mathbb{Q}}
\newcommand{\Q}{\mathbb{Q}}
\newcommand{\Z}{\mathbb{Z}}

\newcommand{\BR}{\mathbb{R}}

\renewcommand{\a}{\alpha}

\newcommand{\et}{\text{\'et}}

\newcommand{\fl}{\mathrm{fl}}
\newcommand{\p}{\varphi}

\newcommand{\BZ}{\mathbb{Z}}
\newcommand{\A}{\mathbb{A}}

\begin{document}

\date{}


\baselineskip 20pt
\setcounter{equation}{0}
\pagestyle{plain}
\pagenumbering{arabic}

\title{On the flat cohomology of binary norm forms}

\author{Rony A.~Bitan and Michael M.~Schein}
\thanks{This work was supported by grant 1246/2014 from the Germany-Israel Foundation.  The first author was also supported by a Chateaubriand Fellowship of the Embassy of France in Israel, 2016.}

\begin{abstract}
Let $\CO$ be an order of index $m$ in the maximal order of a quadratic number field $k=\Q(\sqrt{d})$.
Let $\Ob_{d,m}$ be the orthogonal $\Z$-group of the associated norm form $q_{d,m}$.  
We describe the structure of the pointed set $H^1_\fl(\Z,\Ob_{d,m})$, 
which classifies quadratic forms isomorphic (properly or improperly) to $q_d$ in the flat topology. 
Gauss classified quadratic forms of fundamental discriminant and showed that 
the composition of any binary $\Z$-form of discriminant $\Dl_k$ with itself belongs to the principal genus.  
Using cohomological language, we extend these results to forms of certain non-fundamental discriminants.     
\end{abstract}

\maketitle{}

\markright{Flat cohomology of binary norm forms}

\section{Introduction} 
Let $q$ be an integral binary quadratic form, namely a map $q: \Z^2 \to \Z$ represented by a symmetric matrix  
$B_q$ 
satisfying $q(x,y) = (x,y) B_q (x,y)^t = ax^2 + bxy + cy^2$, where $a,b,c \in \Z$. 
For brevity we write $q=(a,b,c)$. 
{We will underline schemes defined over $\Sp \Z$, omitting the underline for their generic fibers.}    
Any change of variables by $A \in \un{\textbf{GL}}_2(\Z)$ gives rise to an isomorphic form $q'=q \circ A$  
represented by the congruent matrix $B_{q'} = A B_q A^t$. 
In particular, if $q = q \circ A$, then $A$ is said to be an \emph{isometry} of $q$.    
It is called \emph{proper} if $A \in \un{\textbf{SL}}_2(\Z)$. 
The \emph{discriminant} of $q$ is the integer $\Dl(q) = b^2-4ac = -4\det(B_q)$;  
it is independent of the choice of the basis of $\Z^2$ since $\det(A)=\pm 1$ for any $A \in \un{\textbf{GL}}_2(\Z)$. 
Given an integer $n \in \Z$, a natural and very classical problem is to describe the set of equivalence classes
\begin{equation} \label{equ:def.cl.plus}
 \cl^+(n) := \{ q : \Dl(q)=n \} \big / \un{\textbf{SL}}_2(\Z). 
\end{equation}

Consider the quadratic number field $k = \Q(\sqrt{d})$, where $d \not\in \{ 0, 1 \}$ is a square-free integer.  
Denote its discriminant by $\Delta_k$ and the norm map by $\Nr:k^\times \to \Q^\times$.   
Fixing an integral basis $\{ 1, \omega \}$ of the ring of integers $\CO_k$, associate to $k$ the \emph{norm $\Z$-form} $q_d(x,y) := \Nr(x + y \om)$. 
Then $\Dl(q_d) = \Dl_k$ is a {fundamental discriminant}. 
As $\CO_k$ is a Dedekind domain, its {narrow ideal class group} $I_k/P^+_k$ coincides with its {(narrow) Picard group} $\mathrm{Pic}^+(\CO_k)$.    
If $d < 0$, write $\cl^+(\Dl(q_d))'$ for the restriction of $\cl^+(\Dl(q_d))$ to only \emph{positive} definite forms, 
namely those for which $a,c > 0$.  If $d > 0$, define $\cl^+(\Dl(q_d))' = \cl^+(\Dl(q_d))$.
Gauss, in his {\emph{Disquisitiones Arithmeticae}}~\cite{Gau}, proved  
the following bijection of pointed sets: 
{$$ \cl^+(\Dl(q_d))' \cong \mathrm{Pic}^+(\CO_k) \ : \ \left[ (a,b,c) \right] \mapsto 
\left [ \left \la a,\fc{b - F\sqrt{d}}{2} \right \ra \right] \ \text{where} \  
F = \left \{ \begin{array}{l l} 
2 & d \equiv 2,3 \, (\text{mod}~4) \\ \\
1 & d \equiv 1 \, (\text{mod}~4). \\
\end{array}\right.  $$}

The main aim of this paper is to describe the sets
$$ \cl(n) := \{ q : \Dl(q) = n \} \big / \un{\textbf{GL}}_2(\Z), $$
in terms of geometric invariants of orders in quadratic number fields.  This extends the question considered by Gauss in three ways: we consider all quadratic forms, and not only the positive definite ones; we consider all isometries, and not only proper ones; we consider discriminants $n$ that are not fundamental.

A modern perspective on these classical ideas, used in the 1980's by Ono~\cite{Ono} for number fields 
and extended by Morishita~\cite{Mor} to general global fields,   
identifies $\cl^+(\Dl(q_d))$ with $H^1_\fl(\Z,\Od^+)$, where $\Od^+$ is the \emph{special orthogonal group} of $q_d$. 
This flat cohomology, which {\emph{a priori}} is a pointed set but is an abelian group since $\Od^+$ is commutative, classifies all integral binary forms that are locally isomorphic to $q_d$ in the flat (i.e.~fppf) topology modulo proper isometries.  
Analogously, the first Nisnevich cohomology set classifies forms in the principal genus of $q_d$.  
We extend this approach to arbitrary quadratic orders $\CO \subseteq \CO_k$ 
and obtain a classification, in terms of the Picard group $\Pic \CO$,
of isomorphism classes (not just proper isomorphism classes) of integral forms that are locally isomorphic in the flat topology 
to the norm form associated with $\CO$. 

\mk

\mk

\subsection{Organization of the paper} 
We briefly describe the structure of the paper and point out its main results. 
Sections~\ref{Pre} and~\ref{sec:class.set} recall the basic notions we will use, most notably the norm torus associated to an order.  Section~\ref{quadratic forms} defines the orthogonal and special orthogonal groups $\OV$ and $\OV^+$ associated to a quadratic form $q$.  If $q$ is degenerate over $\Z$, the orthogonal group $\OV$ need not be flat over $\Z$.  Thus we work instead with $\widetilde{\Ob}_q$, the schematic closure in $\OV$ of the generic point.  We obtain an identification (Lemma~\ref{N = O+}) of the special orthogonal group of a norm form of an order $\CO$ (with respect to a fixed $\Z$-basis of $\CO$) with the norm torus of $\CO$.  Finally, we let $\CO_{d,m}$ denote the unique order of index $m$ in the maximal order of $k = \Q(\sqrt{d})$ and fix $\Z$-bases of the orders $\CO_{d,m}$.  There is a natural bifurcation into two cases: either the norm form $q_{d,m} := q_{\CO_{d,m}}$ is diagonal for a suitable choice of basis (Case (II)) or not (Case (I)).  Case (I) holds when $d \equiv 1 \, \mathrm{mod} \, 4$ and $m$ is odd, whereas Case (II) covers all other instances.

Section~\ref{sec:main}, the heart of the paper, starts by determining (Proposition~\ref{prop:quotient}) the structure of the quotient $\widetilde{\Ob}_{q_{d,m}} / \Ob_{q_{d,m}}^+$, which is always a finite flat group scheme of order two.  The proof is short and relies on the theory of finite flat group schemes.  For comparison, in an appendix to the paper we provide a more classical proof that writes down defining polynomials of $\widetilde{\Ob}_{q_{d,m}}$ and $\Ob_{q_{d,m}}^+$.  We then turn to studying the pointed set $H^1_\fl(\Z, \Ob_{q_{d,m}}) = H^1_\fl(\Z, \widetilde{\Ob}_{q_{d,m}})$.  In Case (I) it is canonically identified with $H^1_\fl(\Z, \Ob_{q_{d,m}}^+)$, whereas in Case (II) it also contains classes of forms of discriminant $-\Delta(q_{d,m})$.  This is shown in Proposition~\ref{H1 1 mod 4} and Lemma~\ref{H1 2,3 mod 4}, respectively.  From this we can study the sets $\cl(n)$ for many discriminants $n$.  The following is the content of Propositions~\ref{pro:cl.case.i} and~\ref{pro:cl.case.ii}:

\begin{prop}
Let $D \in \Z$ be an integer such that $D \equiv 0 \, \mathrm{mod} \, 4$ or $D \equiv 1 \, \mathrm{mod} \, 4$.  Suppose further that $D$ is not a perfect square and not of the form $D = -3 \cdot 4^m$ for some $m \geq \mathbb{N}_0$.  Then
$$
\cl(D) = \begin{cases}
\cl^+(D) = H^1_\fl(\Z, \OV^+) &: D \equiv 1 \, \mathrm{mod} \, 4 \\
\cl^+(D) / \sim \, = H^1_\fl(\Z, \OV^+)/\sim &: D \equiv 0 \, \mathrm{mod} \, 4,
\end{cases}
$$
where the explicit quadratic form $q$ is the norm form of a quadratic order with respect to one of our explicit bases, and the equivalence relation $\sim$ is given by $[ax^2 + bxy + cy^2] \sim [ax^2 - bxy + cy^2]$.
\end{prop}

The relation stated here between quadratic forms and flat cohomology fails for discriminants of the form $-3 \cdot 4^m$; see Remark~\ref{rem:bad.discriminants}.  Along the way we study a number of explicit examples.  
For any square-free $d \neq 0,1$ we show in Theorem~\ref{H1 Nd and Pic+} that
$$
\cl^+(\Delta_{\Q(\sqrt{d})}) \cong \{\pm 1\}^{\mu(d)} \times \mathrm{Pic}^+(\CO_k), \ \text{where} \  
\mu(d) =
\left \{ 
\begin{array}{l l}
1 & d<0 \\ 
0 & d>0. 
\end{array}\right. $$   
This is a straightforward extension of Gauss' proper classification to all forms, not just the positive definite ones.  More generally, our analysis of Case (II) leads to an extension of the classification to many cases in which $4D$ is not a fundamental discriminant. Theorem~\ref{Gauss Z[root(d)]}, another classical theorem that we prove with new methods, states that if $D$ is any integer that is not a perfect square and not of the form $D = -3 \cdot 4^m$, then
$$ \cl^+(4D) \cong \{\pm 1\}^{\tilde{\ve}(D)} \times \mathrm{Pic}(\Z[\sqrt{D}]), $$
where 
\begin{align*} \tilde{\ve}(D) &=
\left \{ \begin{array}{l l}
0 & D > 0 \ \text{and} \ \mathrm{Nr}(\Z[\sqrt{D}]^\times) = \{ \pm 1 \} \\
1 & \text{otherwise.} \\
\end{array}\right. 
\end{align*} 
Note that $\Z[\sqrt{D}]$ is not a Dedekind domain in general, and that Theorem~\ref{Gauss Z[root(d)]} remains true for discriminants of the form $-3 \cdot 4^m$; it is our proof that fails.
%
Recall that $k = \Q(\sqrt{d})$.  In Proposition~\ref{H1 cardinality}, we express the cardinality $|H^1_\fl(\Z, \Ob_{q_{\CO_k}})|$ in terms of the narrow class numbers of $\Q(\sqrt{d})$ and $\Q(\sqrt{-d})$.  These results do not allow us to deduce any interesting relation between the two narrow class numbers.
Finally, we show in Corollary~\ref{exponent 2} that any $\Ob_{q_{d,m}}^+$-torsor, tensored with itself, 
belongs to the principal genus of $q_{d,m}$. 
This may be viewed as as extension, in the language of cohomology, of another classical theorem of Gauss. 

\bk \noindent

{\bf Acknowledgements:} 
The authors thank J.~Bernstein, B.~Conrad, P.~Gille, B.~Kunyavski\u\i, and B.~Moroz 
for valuable discussions concerning the topics of the present article.  
They are grateful to the anonymous referee for a careful reading of the article and helpful comments that have clarified the exposition.

\bk

\section{Preliminaries} \label{Pre}
Let $k / \Q$ be a finite Galois extension with Galois group $\G = \text{Gal}(k/\BQ)$ and degree $n = [k : \Q]$.   
Let $\un{\BG}_m$ 
and $\un{\textbf{GL}}_n$ denote the multiplicative and general linear $\Z$-groups, respectively. 
Recall that an {\emph{order}} in $\mathcal{O}_k$ is a subring that has the maximal rank $n$ as a $\Z$-lattice.
Fix a $\Z$-basis $\Om = \{\om_1,...,\om_n\}$ for an \emph{order} 
$\CO_\Om \subseteq \CO_k$.
The Weil restriction of scalars $\un{R}_\Om = \Res_{\CO_\Om/\Z}(\un{\BG}_m)$ is an $n$-dimensional algebraic $\Z$-torus 
that admits an isomorphism $\rho: \un{R}_\Om(\Z) \simeq \CO_\Om^\times $~\cite[\S 7.6]{BLR}.  
The natural action of $\rho(\un{R}_\Om(\Z))$ on $\CO_\Om$  
yields a canonical embedding of $\un{R}_\Om$ in 
$\text{Aut}(\CO_\Om) = \un{\textbf{GL}}(\CO_\Om)$, depending only on the order $\CO_\Om$ and not on the $\Z$-basis $\Om$.  
The choice of $\Om$ provides 
an embedding $\iota: \un{R}_\Om \hookrightarrow \un{\textbf{GL}}_n$.   
The composition of $\iota$ 
with the determinant gives a map $\un{R}_\Om \to \un{\BG}_m$ that we abusively denote\footnote{Note that the map $\det$ depends only on the order $\CO_\Om$ and not on the choice of basis $\Om$. 
Indeed, the constructions of this and the following two sections, up to the explicit matrix realizations of~\eqref{qd} and~\eqref{Ad}, 
are independent of the choice of $\Om$.} 
{{$\det$}.}
Then $\mathrm{Nr}(\a) = \det(\iota(\rho^{-1}(\a)))$ for all $\a \in \CO_\Om^\times$, where $\mathrm{Nr}: k^\times \to \Q^\times$ is the usual norm map; see Exercise 9(c) of~\cite[Section II.5]{Bou}. 
We get a short exact sequence of commutative $\Z$-group schemes where the quotient map is faithfully flat in the sense of~\cite[0.6.7.8]{EGAI}:  
\begin{equation} \label{exact Z sequence}
1 \to \N_\Om \to \un{R}_\Om \xrightarrow{\det} \un{\BG}_m \to 1. 
\end{equation}
The generic fibers of the elements of this sequence are the norm torus $N=\Res^{(1)}_{k/\Q}(\BG_m)$, 
the Weil torus $R=\Res_{k/\Q}(\BG_m)$, 
and the multiplicative $\Q$-group $\BG_m$, respectively. 
Their fibers at any prime $p$ are denoted by $(\N_\Om)_p$, $(\un{R}_\Om)_p$ and $(\un{\BG}_m)_p$, respectively,
while their reductions are overlined.  
We omit the subscript $\Om$ when $\CO_\Om$ is the maximal order $\CO_k$.

\bk

While $\un{\BG}_m$ and $\un{R}_\Om$ are smooth over $\Sp \Z$, the kernel $\N_\Om$ need not be smooth, 
in that it may have a non-reduced reduction at some prime. 
However, $\un{N}_\Om$ is the kernel of a morphism of $\CO_\Om$-tori and thus
is of multiplicative type.  In particular, it is faithfully flat and affine.   
So instead of using \'etale cohomology, we shall restrict ourselves to flat cohomology. 


%

Applying flat cohomology to~\eqref{exact Z sequence}
gives rise to a long exact sequence of pointed sets; see \cite[III, Proposition~3.3.1.(i)]{Gir}: 
\begin{equation} \label{Z-points}
1 \to \N_\Om(\Z) \to \un{R}_\Om(\Z) \cong \CO_\Om^\times \stackrel{\Nr}{\longrightarrow} \{ \pm 1 \} 
  \to H^1_\fl(\Z,\N_\Om) \to H^1_\fl(\Z,\un{R}_\Om) \to H^1_\fl(\Z,\un{\BG}_m) = \Pic \Z = 0.  
\end{equation}
Now $\CO_\Om$ is finite and torsion-free over $\Z$, hence flat.
By Shapiro's Lemma~\cite[XXIV,~Prop.~8.2]{SGA3},  
we have $ H^1_\fl(\Z,\un{R}_\Om) \cong H^1_\fl(\CO_\Om,\un{\BG}_{m,\CO_\Om}) = \Pic \CO_\Om$.  
Thus~\eqref{Z-points} can be rewritten as
\begin{equation} \label{N and Pic(Ok)}
1 \to \{\pm 1\}/\Nr(\CO_\Om^\times) \to H^1_\fl(\Z,\N_\Om) \to \Pic (\CO_\Om) \to 1.   
\end{equation} 
The maximal order $\CO_k$ is a Dedekind domain whose Picard group coincides with the ideal class group of $k$. 
The set $\{\pm 1\}/\Nr(\CO_k^\times)$ is equal to the zero-Tate cohomology set 
$H_T^0(\G,\CO_k^\times)$~\cite[Example~1]{Ono}. 
Thus, in the case $\CO_\Omega = \CO_k$, we deduce an isomorphism of finite groups
\begin{equation} \label{Pic Ok}
H^1_\fl(\Z,\N) / H_T^0(\G,\CO_k^\times) \cong \Pic \CO_k. 
\end{equation}
If $n$ is odd, then $\Nr(-1)=(-1)^n=-1$. 
Therefore $H^1_\fl(\Z,\N) \cong \Pic \CO_k$  
and it follows that
\begin{equation} \label{n>2}
h_k = |H^1_\fl(\Z,\N)|.
\end{equation}
In the quadratic case $n = 2$, we have $k=\BQ(\sqrt{d})$ for some square-free integer $d \not\in \{ 0,1 \}$.  
If $\CO_\Omega$ is the maximal order $\CO_k$, we set $h_d$ and $\un{N}_d$ to be the class number $h_k$ and the $\Z$-group $\N$, respectively.   
Then~\eqref{N and Pic(Ok)} implies that 
\begin{equation} \label{ve quadratic}
|H^1_\fl(\BZ,\N_d)| = h_d \cdot 2^{\ve(d)},
\end{equation}
where~\cite[\S 5,Example~2]{Ono}:
\begin{align} \label{equ:def.epsilon}
\ve(d) &=
\left \{ \begin{array}{l l}
0 & d > 0 \ \text{and} \ \mathrm{Nr}(\CO_k^\times) = \{ \pm 1 \} \\
1 & \text{otherwise.} \\
\end{array}\right. 
\end{align} 
Let $\mathrm{Pic}^+(\CO_k)$ be the narrow class group of $k$ and let $h_d^+$ denote its cardinality.  
Then $h_d^+ = h_d$ unless $d>0$ and $\Nr(\CO_k^\times) =\{1\}$, in which case $h_d^+ = 2 h_d$. 
Now~\eqref{ve quadratic} implies 
\begin{align} \label{mu quadratic}
|H^1_\fl(\BZ,\N_d)| = h_d^+ \cdot 2^{\mu(d)}, \ \ 
\mu(d) &:=
\left \{ \begin{array}{l l}
1 & d < 0 \\
0 & d > 0. \\
\end{array}\right. 
\end{align} 
Hence computing the narrow class number $h_d^+$ 
is equivalent to determining $|H^1_\fl(\BZ,\N_d)|$. 

\bk

\section{The class set of the norm torus} \label{sec:class.set}
Let $\un{G}$ be an affine flat group scheme defined over $\Sp \Z$ with generic fiber $G$.  
We denote by $\un{G}_p$ the $\Z_p$-scheme obtained from $\un{G}$ by the base change $\Sp \Z_p \to \Sp \Z$.     
For a global field $F$, recall that the adelic group $\un{G}(\A_F)$ is the restricted product of the groups $G(F_v)$, where $F_v$ is the completion of $F$ at a place $v$.  We write $\un{G}(\A)$ for $\un{G}(\A_\Q)$.
As in~\cite[\S 1.2]{Bor}, consider its subgroup $\un{G}(\A_\infty) = G(\BR) \times \prod_p \un{G}(\Z_p)$. 
 
\begin{definition}  
The \emph{class set} of $\un{G}$ is the set of double cosets  
$\mathrm{Cl}_\iy(\un{G}) := \un{G}(\A_\iy) \backslash \un{G}(\A) / G(\BQ)$.  
This set is finite (\cite[Thm.~5.1]{Bor}) and its cardinality, denoted by $h(\un{G})$, is called the \emph{class number} of ~$\un{G}$.  
\end{definition}

\begin{definition} \label{Sha}
Let $S$ be a finite set of places in $\BQ$. 
The \emph{first Tate-Shafarevich set} of $G$ over $\BQ$ relative to $S$ is
$$ \Sh^1_S(\BQ,G) := \ker\left[H^1(\Q,G) \to \prod_{v \notin S} H^1({\Q}_v,G_v) \right]. $$
When $S = \varnothing$, we simply write $\Sh^1(\BQ,G)$.   
\end{definition}

As $\un{G}$ is affine, flat and of finite type,   
Y.~Nisnevich has shown~\cite[Theorem~I.3.5]{Nis} that it admits an exact sequence of pointed sets
\begin{equation} \label{Nis sequence}
1 \to \text{Cl}_\iy(\un{G}) \to H^1_\fl(\Z,\un{G}) \to H^1(\Q,G) \times \prod_p H^1_\fl({\Z}_p,\un{G}_p)   
\end{equation}
whose left exactness reflects the fact that $\text{Cl}_\iy(\un{G})$ is the set of $\Z$-forms of $\un{G}$ 
that are isomorphic to $\un{G}$ over $\Q$ and over ${\Z}_p$ for all $p$.     
If $H^1_\fl(\BZ_p,\un{G}_p)$ injects into $H^1(\BQ_p,G_p)$ for all $p$, 
then, as in~\cite[Cor.~I.3.6]{Nis}, the sequence~\eqref{Nis sequence} simplifies to
\begin{equation} \label{Nis short}
1 \to \text{Cl}_\iy(\un{G}) \to H^1_\fl(\Z,\un{G}) \to H^1(\Q,G). 
\end{equation}
More precisely, there is an exact sequence of pointed sets (cf.~\cite[Corollary A.8]{GP}) 
\begin{equation} \label{Nis short right exact}  
1 \to \mathrm{Cl}_\iy(\un{G}) \to H^1_\fl(\BZ,\un{G}) \to B \to 1  
\end{equation}
in which 
$$ B = \left\{ [\g] \in H^1(\Q,G): \forall p, [\g \otimes {\Z}_p] \in \Im \left(H^1_\fl({\Z}_p,\un{G}_p) \to H^1({\Q}_p,G_p) \right) \right\}. $$

\bk

Let $k / \Q$ be a finite Galois extension as in the previous section.
Let $p$ be a rational prime, and let $P$ be a prime of $k$ dividing $p$. 
Write $\Q_p$ and $k_P$ for the localizations of $\Q$ at $p$ and of $k$ at $P$, respectively,    
noting that $k_P$ is independent of the choice of $P$, up to isomorphism.
Observe that $k \otimes_\Q \Q_p \cong k_P^r$, where $r$ is the number of primes of $k$ dividing $p$. 
The norm map $\Nr : k \to \Q$ induces a map $\mathrm{Nr}: k \otimes_\Q \Q_p \to \Q_p$; 
under the isomorphism above this corresponds to the product of the norm maps $
\mathrm{N}_{k_P / \Q_p}$ on the components. 
Similarly, $\mathcal{O}_k \otimes_\Z \Z_p \simeq \mathcal{O}_{k_P}^r$.  Write $U_P$ for $\mathcal{O}_{k_P}^\times$.

Fixing a $\Z$-basis $\Omega$ of an order $\CO_\Omega \subset \CO_k$ as in the previous section and applying flat cohomology to the short exact sequence of flat $\Z_p$-groups 
$$ 1 \to \N_p \to \un{R}_p \to (\un{\BG}_m)_p \to 1 $$
yields the exact and functorial sequence   
$$ 1 \to \N_p(\Z_p) \to \un{R}_p(\Z_p) \cong U_P^r \xrightarrow{\mathrm{Nr}} \BZ_p^\times \to H^1_\fl(\Z_p,\N_p) \to 1, $$
since $H^1_\fl(\Z_p, \un{R}_p)$ is the Picard group of a product of local rings and thus vanishes.  
We deduce an isomorphism $H^1_\fl(\BZ_p,\N_p) \cong \BZ_p^\times / \mathrm{Nr}(U_P^r) = \BZ_p^\times / \mathrm{N}_{k_P / \Q_p}(U_P)$. 
Applying Galois cohomology to the short exact sequence of $\BQ_p$-groups
$$ 1 \to N_p \to R_p \to (\BG_m)_p \to 1 $$
gives rise to the exact sequence of abelian groups
$$ 1 \to N_p(\BQ_p) \to R_p(\BQ_p) \cong (k_P^\times)^r \xrightarrow{\mathrm{Nr}} \BQ_p^\times \to H^1(\BQ_p,N_p) \to 1, $$
where the rightmost term vanishes by Hilbert's Theorem~90.
Hence {we may again deduce a functorial isomorphism} $H^1(\BQ_p,N_p) \cong \BQ_p^\times / \mathrm{N}_{k_P / \Q_p}(k_P^\times)$. 
Note that $U_P$ is compact and thus $\mathrm{N}_{k_P / \Q_p}(U_P)$ is closed in $\Q_p^\times$. 
Only units have norms that are units, so we obtain an embedding of groups: 
\begin{equation} \label{embedding of local N}
H^1_\fl(\BZ_p,\N_p) \cong \BZ_p^\times / \mathrm{N}_{k_P / \Q_p}(U_P) \hookrightarrow \BQ_p^\times / \mathrm{N}_{k_P / \Q_p}(k_P^\times) \cong H^1(\BQ_p,N_p). 
\end{equation}

\begin{prop} \label{h(N)}
Suppose that $[k:\Q]$ is prime. Let $S_r$ be the set of primes dividing $\Delta_k$. 
Then there is an exact sequence of abelian groups (compare with formula (5.3) in \cite{Mor}):  
$$ 1 \to \text{Cl}_\iy(\N) \to H^1_\fl(\Z,\N) \to \Sh^1_{S_r \cup \{\iy\}}(\BQ,N) \to 1. $$  
\end{prop}

\begin{proof} 
Since $H^1_\fl(\Z_p,\N_p)$ embeds into $H^1(\Q_p,N_p)$ for any prime $p$ by~\eqref{embedding of local N}, 
the $\Z$-group scheme $\N$ admits the exact sequence~\eqref{Nis short right exact}, 
in which the terms are abelian groups as $\N$ is commutative.  
The pointed set $\text{Cl}_\iy(\N)$ is in bijection with the first Nisnevich cohomology set $H^1_{\text{Nis}}(\Z,\N)$ (cf.~\cite[I.~Theorem~2.8]{Nis}), 
which is a subgroup of $H^1_\fl(\Z,\N)$ because any Nisnevich cover is flat. 
Hence the first map is an embedding. 
Since $k/\Q$ has prime degree and so is necessarily abelian, at any prime $p$ the local Artin reciprocity law
implies that
$$ n_p=|\text{Gal}(k_P/\Q_p)|=[\Q_p^\times:\text{N}_{k_P/\Q_p}(k_P^\times)]=|H^1(\Q_p,N_p)|. $$
Furthermore, since $[k:\Q]$ is prime, any ramified place $p$ is totally ramified,   
so $[\Z_p^\times : \mathrm{N}_{k_P / \Q_p}(U_P)] = n_p$ \cite[Theorem~5.5]{Haz}.   
Together with~\eqref{embedding of local N} this means that $H^1_\fl(\Z_p,\N_p)$ 
coincides with $H^1(\Q_p,N_p)$ at ramified primes and vanishes elsewhere. 
Thus the set $B$ of~\eqref{Nis short right exact} consists of classes $[\g] \in H^1(\Q,N)$ 
whose fibers vanish at unramified places.   
This means that $B=\Sh^1_{S_r \cup \{\iy \}}(\BQ,N)$, where $S_r$ is the set of ramified primes of $k / \Q$. 
\end{proof}

\begin{remark} \label{B}
The group $B=\Sh^1_{S_r \cup \{\iy\}}(\Q,N)$ embeds in the group $H^1(\BQ,N)$ by definition. 
But $H^1(\BQ,N) \cong \BQ^\times/ \mathrm{Nr}(k^\times)$, so $B$ has exponent dividing $n$.   
\end{remark}

\bk

\section{Norm forms of orders in quadratic number fields} \label{quadratic forms}
\subsection{Orthogonal groups}
Throughout the rest of this article we will assume that $k$ is a quadratic number field, 
so that $k = \Q(\sqrt{d})$, where $d \not\in \{ 0, 1 \}$ is a square-free integer. 
%
Recall that a \emph{binary integral quadratic form} 
is a homogeneous polynomial of order two in two variables with coefficients in $\Z$: 
$$ q : \Z^2 \to \Z; \  q(x,y) = ax^2+bxy+cy^2,  \ a,b,c \in \Z. $$  
The form $q$ is represented by the symmetric $2 \times 2$ matrix 
$B_q=\left( \begin{array}{cc}
    a    & b/2  \\ 
    b/2  & c  \\ 
	\end{array}\right)$ satisfying $q(x,y) = (x,y) B_q (x,y)^t$. 
We denote $q$ by the triple $(a,b,c)$ and set $\Delta(q) = b^2 - 4ac$.  
Consider the symmetric bilinear form 
$$ \wt{B_q}(u,v) = q(u+v) - q(u) - q(v),$$
for $u,v \in \Z^2$.  Set $\mathrm{disc}(q)$ to be the determinant of the matrix $\wt{B_q}(e_i,e_j)_{1 \leq i,j \leq 2}$, where $e_1 = (1,0)$ and $e_2 = (0,1)$.  We say that $q$ is \emph{non-degenerate} over a $\Z$-algebra $R$ if $\text{disc}(q)$ is invertible in $R$. In particular, $q$ is non-degenerate over $\Z$ when $\mathrm{disc}(q) = \pm 1$ (cf. \cite[\S 2]{Con2}). 
It is easily checked that this matrix is $2B_q$, 
thus $\Delta(q) = -4\det(B_q) = -\text{disc}(q)$. 
We assume $\Delta(q) \neq  0$, so $q$ is non-degenerate over $\Q$.  Observe that $q$ is degenerate over $\Z$ unless $q(x,y) = \pm xy$. 
Two integral forms $q$ and $q'$ are said to be \emph{isomorphic} over a $\Z$-algebra $R$
if there exists an $R$-\emph{isometry} from one form to the other, namely 
a matrix $A \in \un{\textbf{GL}}_2(R)$ such that $q \circ A = q'$. 
If $\det A = 1$, then we say that $A$ gives a {\emph{proper isomorphism}} over $R$ between $q$ and $q'$.

\begin{definition} (\cite[p.303]{Con2}) \label{def:orthogonal.group}
Let $V$ be a free $\Z$-module of rank two and $q: V \to \Z$ a quadratic form with $\mathrm{disc}(q) \neq 0$.
The \emph{orthogonal group} of the quadratic lattice $(V,q)$ is the affine $\Z$-group 
$$ \OV = \{ A \in \textbf{GL}(V) : q \circ A = q \}. $$
\end{definition}


Since $(V,q)$ is not assumed to be non-degenerate over $\Z$, 
we note that $\OV$ may fail to be $\Z$-flat for fiber-jumping reasons~\cite[Section 2]{Con3}.  
We are thus led to restrict our attention to the closed subscheme $\widetilde{\OV} \subset \OV$ defined as the schematic closure in $\OV$ of the generic point.   
Since $(V,q)$ is non-degenerate over $\Q$, and the characteristic of $\Q$ is not $2$,  
we may define the special orthogonal subgroup of the generic fiber $\textbf{O}_q = \OV \otimes_{\Z} \BQ$  
naively as  
$$ \textbf{O}^+_q = \ker[\textbf{O}_q \xrightarrow{\det} \mu_2]. $$  
The analogous definition over $\Z$ is more subtle but is not limited to the non-degenerate case. 
  
\begin{definition} \label{SOV} (\cite[Def.~2.8]{Con3})
The \emph{special orthogonal group} $\OV^+$ of a quadratic lattice
$(V,q)$ is the schematic closure of $\textbf{O}^+_q$ inside $\OV$ (or, equivalently, inside $\widetilde{\OV}$). 
As $\Z$ is Dedekind it is flat.  Indeed, the coordinate ring of the schematic closure, in a $\Z$-scheme, of an affine subscheme of the generic fiber is clearly torsion-free, hence flat over $\Z$.
\end{definition}

\begin{remark} \label{unique closed and flat subgroup}
The $\Q$-group $\textbf{O}_q$ is smooth.  Its open subgroup $\textbf{O}_q^+$ is smooth and connected~\cite[Theorem~1.7(1)]{Con1} and
is the unique such subgroup~\cite[Proposition~3.2]{Con1}. 
By the correspondence between flat closed subschemes of $\OV$  
and closed subschemes of $\textbf{O}_q$ \cite[Cor.~2.8.1]{EGAIV},   
$\OV^+$ is the unique flat and closed subgroup of $\OV$ whose generic fiber is ~$\textbf{O}_q^+$. 
\end{remark}

\begin{definition} 
Let $\Om = \{\om_1,\om_2\} \subset \CO_k$ be a basis of a quadratic order $\CO_\Om \subseteq \CO_k$. 
The \emph{norm form} associated to $\Om$ is the integral quadratic form   
$$ q_\Om(x,y) := \Nr(x\om_1 + y \om_2). $$   
\end{definition}

Let $\un{\textbf{O}}_\Om$ denote the orthogonal group of $q_\Om$.  The choice of the basis $\Omega$ specifies an isomorphism of $\Z$-modules $r : \Z^2 \stackrel{\sim}{\to} \mathcal{O}_\Omega$ given by $r(x,y) = x \omega_1 + y \omega_2$.  In turn, as we observed at the beginning of Section~\ref{Pre}, 
the map $r$ induces 
an embedding $\iota: \un{R}_\Omega \hookrightarrow \un{\textbf{GL}}_2$.

\begin{lem} \label{N = O+}
Let $\Omega$ be a basis of the quadratic order $\CO_\Omega \subseteq \CO_k$.  Then
$\un{\boldsymbol{\mathrm{O}}}_\Om^+ = \N_\Om$. 
\end{lem}

\begin{proof} 
Over $\Q$, consider the map $q_\Om = \Nr \circ r: \Q^2 \to \Q$. 
For any $b \in R_\Om$ and $(x,y) \in \Z^2$ one has
$$ \iota(b) \cdot 
\left( \begin{array}{c}
    x  \\ 
    y  \\ \end{array}\right)
= r^{-1} \left(b \cdot r
\left( \begin{array}{c}
    x  \\ 
    y  \\ \end{array}\right) \right) .$$  
If $b \in N_\Om = \ker ( \Nr:R_\Om \to \BG_{m,\Q} ),$ then $\Nr \circ b = \Nr$ and we obtain an inclusion of $\Q$-groups: 
\begin{align*} 
\textbf{O}^+_\Om &=         \{ A \in \textbf{SL}_2: q_\Om \circ A = q_\Om \} \\ \nonumber 
                 &\supseteq \{ b \in N_\Om:  q_\Om \circ \iota(b) = q_\Om \} \\ \nonumber     
								 &=         \{ b \in N_\Om:  q_\Om \circ r^{-1} \circ b \cdot r = q_\Om \} \\ \nonumber  
			 					 &=         \{ b \in N_\Om:  \Nr \circ b \cdot r = \Nr \circ r \} = \{ b \in N_\Om:  \Nr \circ b = \Nr \} = N_\Om. 
\end{align*}
Since $\textbf{O}^+_\Om$ and $N_\Om$ are both one-dimensional tori, the inclusion is an equality.  Hence $\un{\boldsymbol{\mathrm{O}}}_\Om^+$ and $\N_\Om$ are $\Z$-flat closed subgroups of $\un{\boldsymbol{\mathrm{O}}}_\Om$ with the same generic fiber $\textbf{O}^+_\Om = N_\Om$.  Such an object is unique by Remark~\ref{unique closed and flat subgroup}, so $\un{\boldsymbol{\mathrm{O}}}_\Om^+ = \N_\Om$.
\end{proof}

We note that for the particular bases $\Omega$ used in the sequel, Lemma~\ref{N = O+} can be checked explicitly; see Remark~\ref{rmk:explicit.equality}.

\subsection{Orders in quadratic fields}
Recall that $k = \Q(\sqrt{d})$.  For every $m \in \mathbb{N}$, the maximal order $\CO_k$ contains a unique order $\CO_{d,m}$ of index $m$.  If $\{ 1, \omega \}$ is any $\Z$-basis of $\CO_k$, then $\CO_{d,m}$ is spanned by $\{ 1, m \omega \}$.  We fix the convenient $\Z$-basis $\Omega_{d,m} = \{ 1, \omega_{d,m} \}$ of $\CO_{d,m}$, where 
$$ \omega_{d,m} = 
\begin{cases}
\frac{1 + m \sqrt{d}}{2} &: d \equiv 1 \, \mathrm{mod} \, 4, \, m \, \mathrm{odd} \\
\frac{m}{2} \sqrt{d} &: d \equiv 1 \, \mathrm{mod} \, 4, \, m \, \mathrm{even} \\
m \sqrt{d} &: d \equiv 2,3 \, \mathrm{mod} \, 4.
\end{cases}
$$

Henceforth we denote by $q_{d,m}$ the associated norm form $q_{\Om_{d,m}}$.  We also set $q_d = q_{d,1}$.  Define
\begin{equation} \label{equ:c}
c_{d,m} = 
\begin{cases}
\frac{1 - m^2 d}{4} &: d \equiv 1 \, \mathrm{mod} \, 4, \, m \, \mathrm{odd} \\
\frac{m^2 d}{4}  &: d \equiv 1 \, \mathrm{mod} \, 4, \, m \, \mathrm{even} \\
m^2 d &: d \equiv 2,3 \, \mathrm{mod} \, 4
\end{cases}
\end{equation}
and note that $c_{d,m}$ is always an integer.  For simplicity in long expressions, we will sometimes drop the subscripts and write $c$ for $c_{d,m}$; this should cause no confusion.  We also write $\OVdm$ for $\un{\mathbf{O}}_{q_{d,m}}$, $\widetilde{\Ob}_{d,m}$ for $\widetilde{\Ob}_{d,m}$, etc.  We say that we are in {\bf{Case (I)}} if $d \equiv 1 \, \mathrm{mod} \, 4$ and $m$ is odd, and in {\bf{Case (II)}} otherwise.
Then 
\begin{align} \label{qd}
q_{d,m} = \left \{ \begin{array}{l l}
    (1,1,c_{d,m})   &: \mathrm{Case} \, \mathrm{(I)} \\ 
    (1,0,-c_{d,m})  &: \mathrm{Case} \, \mathrm{(II)}.
\end{array}\right. 
\ \ \text{and } \ \  
B_{q_{d,m}} =  
\left \{ \begin{array}{l l}
\left( \begin{array}{cc}
      1 &  1/2  \\ 
    1/2 &  c_{d,m}  \\ \end{array}\right) &: \mathrm{Case} \, \mathrm{(I)} \\ \\
\left( \begin{array}{cc}
      1 &   0  \\ 
      0 &  -c_{d,m}   \\ \end{array}\right) &: \mathrm{Case} \, \mathrm{(II)}.  
\end{array}\right. 
\end{align} 
Hence 
\begin{align} \label{norms}
\N_{d,m} :=  \N_{\Omega_{d,m}} = 
\left \{ \begin{array}{l l}
\Sp \Z[x,y]/ (x^2 + xy + c_{d,m}y^2-1) &: \ \mathrm{Case} \, \mathrm{(I)}   \\ 
\Sp \Z[x,y]/(x^2-c_{d,m}y^2-1)         &: \ \mathrm{Case} \, \mathrm{(II)}.  
\end{array}\right. 
\end{align}  
The integral matrix realization $\iota(\un{N}_{d,m}(\Z))$ is given by
\begin{align} \label{Ad}
A_d =  
\left \{ \begin{array}{l l}
\left( \begin{array}{cc}
    x & y    \\ 
  -c_{d,m} y & x+y  \\ 
	\end{array}\right), \ \det = 1  &: \mathrm{Case} \, \mathrm{(I)}   \\ \\
\left( \begin{array}{cc}
    x   & y   \\ 
    c_{d,m}y  & x  \\ 
	\end{array}\right), \ \det = 1   &: \mathrm{Case} \, \mathrm{(II)}.  
\end{array}\right. 
\end{align}  
For any pair $(d,m)$, the integral model $\un{N}_{d,m}$ has the generic fiber
\begin{equation*}
N_{d,m} =\N_{d,m} \otimes_{ \Z} \Q = \Sp \Q[x,y]/(x^2-dy^2-1). 
\end{equation*}
Note that $N_{d,m}$ is independent of $m$.

\bk

\section{The flat cohomology of the orthogonal group of a norm form} \label{sec:main}
\subsection{A quotient map}
We have shown in the previous section that the special orthogonal subgroup $\Ob_{d,m}^+$ is a flat closed subgroup of $\Ob_{d,m}$.  Since $\mathrm{Spec} \, \Z$ is one-dimensional, 
the fppf quotient $\Ob_{d,m} / \Ob_{d,m}^+$ is representable by~\cite[Th\'{e}or\`{e}me 4.C]{Ana} and thus has the structure of an affine $\Z$-group scheme.  However, this quotient is not flat.  Instead, we consider the quotient $\un{Q}_{d,m} = \widetilde{\Ob}_{d,m} / \Ob_{d,m}^+$, which is flat over $\Z$ because $\widetilde{\Ob}_{d,m}$ is.  Our first goal in this section is to determine the structure of $\un{Q}_{d,m}$, for which we will use the theory of finite flat group schemes.  An alternative proof, by means of explicitly writing out defining equations, is presented in an appendix at the end of the paper.

Recall that there are only two finite $\BZ$-groups of order two, up to isomorphism,
namely $\un{\Z/2} = \Sp \BZ[t]/(t^2-t)$ and $\un{\mu}_2 = \Sp \BZ[t]/(t^2-1)$;
see, for instance, the Corollary on page 21 of~\cite{TO} for a proof of this fact.
These two $\Z$-groups are locally isomorphic everywhere except at $(2)$, 
in which case $\un{\mu}_2 \otimes_{ \Z} \BF_2$ contains one nilpotent point 
while $\un{\Z/2} \otimes_{\Z} \BF_2$ is reduced and contains two points. 
%
%
%
\begin{prop} \label{prop:quotient}
Let $d \neq 0,1$ be square-free and $m \in \mathbb{N}$.  Then, as $\Z$-group schemes,
$$ \widetilde{\Ob}_{d,m} / \Ob_{d,m}^+ \simeq 
\begin{cases}
\un{\Z / 2} &: \mathrm{Case} \, \mathrm{(I)} \\
\un{\mu}_2 &: \mathrm{Case} \, \mathrm{(II)}.
\end{cases}
$$
\end{prop}
\begin{proof}
Recall from Definition~\ref{def:orthogonal.group} that for any quadratic form $q$, an isometry $A \in \OV$ satisfies $q \circ A = q$ and hence $A B_q A^t = B_q$.  Taking determinants of both sides, we find that $(\det B_q)((\det A)^2 - 1) = 0$.  In particular, if $q$ is non-degenerate over $\Q$, then any $A \in \widetilde{\Ob}_q$ satisfies $(\det A)^2 - 1 = 0$.  Hence the determinant induces a morphism of group schemes $\det: \widetilde{\Ob}_q \to \un{\mu}_2$.  


Since $\widetilde{\Ob}_{d,m}$ is a $\Z$-scheme defined by the polynomial $(\det A - 1)(\det A + 1)$, among others, it is the scheme-theoretic union of the closed $\Z$-subschemes of matrices of determinant $1$ and $-1$.  The former of these is $\Ob_{d,m}^+$, and we denote the latter by $\Ob_{d,m}^-$.  The left translation action on $\widetilde{\Ob}_{d,m}$ of $\left( \begin{array}{cc} 1 & 0 \\ 1 & -1 \end{array} \right)$ in Case (I) and of $\left( \begin{array}{cc} 1 & 0 \\ 0 & -1 \end{array} \right)$  in Case (II) interchanges the subschemes $\Ob_{d,m}^+$ and $\Ob_{d,m}^-$.
Hence all fibers of $\un{Q}_{d,m}$ have rank two, and we conclude that $\un{Q}_{d,m}$ is a finite flat affine $\Z$-group scheme of order two.  Thus it is isomorphic either to $\un{\Z / 2}$ or to $\un{\mu}_2$.  To distinguish between the two possibilities, it suffices to determine $\un{Q}_{d,m} \otimes_{\Z} \mathbb{F}_2$.
Since the reduction of $\mathrm{diag}(1,-1)$ modulo $2$ is the same as that of the identity matrix, we see that $\un{Q}_{d,m} \otimes_{\Z} \mathbb{F}_2$ contains only one point in Case (II).  In Case (I), on the other hand, it is apparent from~\eqref{Ad} and Lemma~\ref{N = O+} that the reduction modulo $2$ of the matrix $\left( \begin{array}{cc} 1 & 0 \\ 1 & -1 \end{array} \right)$ does not lie in that of $\un{N}_{d,m}(\Z) = \Ob_{d,m}^+(\Z)$.  Thus $\un{Q}_{d,m} \otimes_{\Z} \mathbb{F}_2$ has two points in this case.  We conclude that $\un{Q}_{d,m} = \un{\Z / 2}$ in Case (I) and $\un{Q}_{d,m} = \un{\mu}_2$ in Case (II), as claimed.
\end{proof}
\subsection{Twisted forms}
In this section we briefly recall the construction of a twisted form (in the fppf topology) of a group scheme $\un{G}$ defined over $\Sp R$, where $R$ is any unital commutative ring.  
Let $P$ be a representative of a class in $H^1_\fl(R,\un{G})$, i.e.~a $\un{G}$-torsor.   
Viewing $\un{G}$ as a $\un{G}$-torsor acting on itself by conjugation, we may consider the affine $R$-group scheme ${^P}\un{G}$ given by
the quotient of $P \times_R \un{G}$ by the $\un{G}$-action $s \cdot (p,g) = (ps^{-1},sgs^{-1})$.    
This is an inner form of $\un{G}$, 
called the \emph{twist} of $\un{G}$ by $P$.
It is isomorphic to $\un{G}$ in the fppf topology, 
and the map $\un{G} \mapsto {^P}\un{G}$ defines a bijection of pointed sets
$\theta_P: H^1_\fl(R,\un{G}) \to H^1_\fl(R,{^P}\un{G})$; see~\cite[\S 2.2, Lemma 2.2.3]{Sko} and the nearby Examples 1 and 2.   



\begin{remark} \label{rem:inner.form}
Fix a quadratic $\Z$-form $q$ of rank $n$ with algebraic automorphism group $\textbf{Aut}(q) := \OV$ defined over $\Sp \Z$.  
Any quadratic $\Z$-form $q'$ of rank $n$ corresponds to an $\OV$-torsor by 
$$ q' \mapsto \textbf{Iso}(q,q'). $$
This is a special case of a general framework due to Giraud; see~\cite[Proposition~2.2.4.5]{CF} for details.
It follows that $\un{\textbf{O}}_{q'}$ is an inner form of $\OV$, namely its twist by $\textbf{Iso}(q,q')$.  In terms of the classification of $\OV$-torsors by cohomology classes, this corresponds to changing the distinguished base point of $H^1_\fl(\Z,\OV)$.
\end{remark}


\begin{remark}
For any flat, i.e.~torsion-free, $\Z$-algebra $R$ the inclusion $\widetilde{\OV}(R) \subseteq \OV(R)$ of $R$-points is an equality.  This implies a canonical isomorphism of pointed sets $H^1_\fl(\Z,\widetilde{\un{{\textbf{O}}}_q}) = H^1_\fl(\Z,\un{\textbf{O}}_q)$.  In the sequel we identify these sets without further comment.
\end{remark}

\begin{lem} \label{TFAE}
Let ${^P} \OV$ be the twisted form of $\OV$ by an $\OV^+$-torsor $P$ and let $\un{Q}:=\widetilde{\OV}/\OV^+$. 
Then the following are equivalent: 
\begin{itemize}
\item[(1)] The push-forward map $H^1_\fl(\Z,\OV^+) \xrightarrow{i_*} H^1_\fl(\Z,\widetilde{\OV})$ is injective. 
\item[(2)] The quotient map $\widetilde{{^P} \OV}(\Z) \xrightarrow{\pi} \un{Q}(\Z)$ is surjective for any $[P] \in H^1_\fl(\Z,\OV^+)$. 
\item[(3)] The $\un{Q}(\Z)$-action on $H^1_\fl(\Z,\OV^+)$ is trivial. 
\end{itemize} 
\end{lem}

\begin{proof}
By the correspondence discussed above between quadratic $\Z$-forms and $\OV$-torsors,
the inner form ${^P} \OV$ of $\OV$ is the orthogonal group of some quadratic $\Z$-form $q^\prime$. 
Consider the commutative diagram with exact rows (cf.~\cite[Lemme~III.3.3.4]{Gir})
$$\xymatrix{
\widetilde{\OV}(\Z)    \ar[r]^{\pi} & \un{Q}(\Z) \ar[r] & H^1_\fl(\Z,\OV^+)    \ar[d]_{\cong}^{\theta_P} \ar[r]^{i_*} & H^1_\fl(\Z,\widetilde{\OV}) \ar[d]^{\theta_P}_{\cong} \\
\widetilde{{^P} \OV}(\Z) \ar[r]^{{^P}\pi} & \un{Q}(\Z) \ar[r] & H^1_\fl(\Z,{^P} \OV^+) \ar[r]^{i_*'}                          & H^1_\fl(\Z,\widetilde{{^P} \OV}), 
}$$
where the top row arises by applying flat cohomology to the sequence $1 \to \un{\mathbf{O}}_q^+ \to \widetilde{\un{\mathbf{O}}_q} \to \un{Q} \to 1$, whereas the bottom row comes from the analogous sequence for $q^\prime$ and the maps $\theta_P$ are the induced twisting bijections defined above. \\
%
$(1) \Leftrightarrow (2)$: The map $i_*$ is injective if any class $[P]$ of $\OV^+$-torsors is the unique pre-image of $i_*([P]) \in H^1_\fl(\Z,\widetilde{\OV}) = H^1_\fl(\Z,\OV)$.   
By commutativity of the diagram, this is equivalent to the distinguished point in $H^1_\fl(\Z,{^P} \OV^+)$ being the unique pre-image of its image,     
for any choice of a twisted form ${^P} \OV^+$ of $\OV^+$, i.e. to the triviality of $\ker(i_*')$ for any $\OV^+$-torsor $P$. 
By exactness of the rows, this is condition~$(2)$.  \\
$(1) \Leftrightarrow (3)$: 
By~\cite[Prop.~III.3.3.3(iv)]{Gir}, $i_*$ induces an injection of $H^1_\fl(\Z,\OV^+) / \un{Q}(\Z)$ into $H^1_\fl(\Z,\widetilde{\OV})$. 
Thus $i_*:H^1_\fl(\Z,\OV^+) \to H^1_\fl(\Z,\widetilde{\OV})$ is injective if and only if $\un{Q}(\Z)$ acts on $H^1_\fl(\Z,\OV^+)$ trivially.      
\end{proof}

\subsection{Case (I)}
Recall that Case (I) means that $d \equiv 1 \, \mathrm{mod} \, 4$ and $m$ is odd.
In this case the quotient $\widetilde{\Ob}_{d,m} / \Ob_{d,m}^+$ is $\un{\Z / 2}$ by Proposition~\ref{prop:quotient}, the quotient map being the Dickson morphism $D_{q_{d,m}}$.
%

\begin{prop} \label{H1 1 mod 4}
In Case (I), there is an equality of abelian groups $H^1_\fl(\Z,\N_{d,m}) = H^1_\fl(\Z,\un{\bf{O}}_{d,m})$. 
\end{prop}   

\begin{proof}
Applying flat cohomology to the short exact sequence of flat $\Z$-schemes
\begin{equation*}
1 \to \Ob_{d,m}^+ \to \widetilde{\Ob}_{d,m} \to \un{Q}_{d,m} \to 1 
\end{equation*}
gives rise to a long exact sequence of pointed sets
\begin{equation} \label{LES}
\widetilde{\Ob}_{d,m}(\Z) \to \un{Q}_{d,m}(\Z) \to H^1_\fl(\Z,\Ob_{d,m}^+) \stackrel{i_*}{\to} H^1_\fl(\Z,{\Ob}_{d,m}) \stackrel{\delta}{\to} H^1_\fl(\Z,\un{Q}_{d,m}).  
\end{equation}   

We will show that $i_*$ is an isomorphism; the claim then follows by Lemma~\ref{N = O+}.  
Since $\un{\Z/2}$ is smooth, the rightmost term in~\eqref{LES} coincides with $H^1_\et(\Z,\un{\Z/2})$ by~\cite[Corollaire~VIII.2.3]{SGA4}, 
{which classifies \'etale quadratic covers of $\Z$ (cf. \cite[Chap.III, Prop.4.1.4]{Knus}). 
As no such non-trivial cover exists, $\dl$ is trivial, and $i_*$ is surjective.} 

To prove that $i_*$ is injective, it suffices by Lemma~\ref{TFAE} to show that the map 
$\widetilde{{^P}\Ob}_{d,m}(\Z) \rightarrow \un{\Z/2}(\Z)$ is surjective for all $[P] \in H^1_\fl(\Z,\Ob_{d,m}^+)$. 
This is true when $[P]$ is the distinguished class, since we verified explicitly in the course of the proof of Proposition~\ref{prop:quotient} that $D_{q_{d,m}} : \widetilde{\Ob}_{d,m}(\Z) \to \un{\Z / 2}(\Z)$ is surjective.  In fact, this implies that the determinant map $ \widetilde{\Ob}_{d,m}(\Z) \to \un{\mu}_2(\Z) = \{ \pm 1 \}$ is surjective on $\Z$-points.  
In general, ${^P}\Ob_{d,m}^+$ is an inner form of $\Ob_{d,m}^+$ by Remark~\ref{rem:inner.form}, 
and thus it is a conjugate of $\Ob_{d,m}^+$ by some $M \in \textbf{GL}_2(R)$ for a finite flat extension $R / \Z$.     
But ${^P}\Ob_{d,m}^+$ is the special orthogonal group of a quadratic $\Z$-form $q$ and $\un{\textbf{O}}_q = M\Ob_{d,m} M^{-1}$.  
Conjugation preserves the determinant, so $\widetilde{\OV} (\Z) \stackrel{\det}{\to} \un{\mu}_2(\Z)$ is surjective, 
and hence so is $D_q: \widetilde{\OV}(\Z) \rightarrow \un{\Z / 2}(\Z)$.
\end{proof}

\subsection{Case (II)}
In this case, we know from Proposition~\ref{prop:quotient} that the determinant map induces an isomorphism $\widetilde{\Ob}_{d,m} / \Ob_{d,m}^+ \simeq \un{\mu}_2$.  
The relevant bit of the long exact sequence~\eqref{LES} is: 
\begin{equation*} 
H^1_\fl(\BZ,\Ob_{d,m}^+) \stackrel{i_*}{\to} H^1_\fl(\BZ,\Ob_{d,m}) \stackrel{\text{disc}}{\surj} H^1_\fl(\BZ,\un{\mu}_2) \cong \{ \pm 1\}.
\end{equation*}    
Here $\text{disc}$, which assigns to any class $[q] \in H^1_\fl(\BZ,\Ob_{d,m})$ the sign of the discriminant of $q$, 
is surjective because $[(1,0,c_{d,m})],[(1,0,-c_{d,m})] \in H^1_\fl(\Z,\Ob_{d,m})$; observe that 
$(1,0,c_{d,m})$ becomes isomorphic to $(1,0,-c_{d,m})$ over $\Z[\sqrt{-1}]$ by the isometry $A = \diag(1,\sqrt{-1})$.

{\begin{lem} \label{H1 2,3 mod 4} 
Suppose that Case (II) holds, so $q=(1,0,-c_{d,m})$, i.e.~$q(x,y) = x^2 - c_{d,m}y^2$.  Then  
\begin{equation} 
H^1_\fl(\Z,\Ob_{d,m}) = H^1_\fl(\Z,\Ob_{d,m}^+)/\un{\mu}_2(\Z) \coprod{H^1_\fl(\Z,\un{\mathbf{O}}^+_{(1,0,c_{d,m})})/\un{\mu}_2(\Z)}, 
\end{equation}
where the non-trivial element of $\un{\mu}_2(\Z)$ maps $[(a,b,c)]$ to $[(a,-b,c)]$.  
\end{lem}}

\begin{proof}
The action of the non-trivial element of $\un{\mu}_2(\Z) = \{ \pm 1\}$ on $H^1_\fl(\Z,\Ob_{d,m}^+)$ is described by~\cite[Remarque~III.3.3.2]{Gir}.
The pre-image of $-1$ under $\det: \widetilde{\Ob}_{q,m} \to \un{\mu_2}$ is $\Ob_{d,m}^- := \diag(1,-1) \Ob_{d,m}^+$. 
A twisted form of $\Ob_{d,m}^+$ by some $\Ob_{d,m}^+$-torsor is $\un{\textbf{O}}^+_{q'}$ where $q'$ is of discriminant $4c_{d,m}$. 
Then the action of $-1$ on $\un{\textbf{O}}^+_{q'}$ is a twist by $\Ob_{d,m}^-$, 
which is equivalent to the twist of $q'=(a,b,c)$ by $\diag(1,-1)$ to $(a,-b,c)$:   
$$  
\left( \begin{array}{cc}
    1 & 0  \\ 
    0 & -1  \\ 
	\end{array}\right) 
	\left( \begin{array}{cc}
    a   & b/2  \\ 
    b/2 & c  \\ 
	\end{array}\right) 
	\left( \begin{array}{cc}
    1 &  0  \\ 
    0 & -1  \\ 
	\end{array}\right)^t 
	=
	\left( \begin{array}{cc}
    a    & -b/2  \\ 
    -b/2 & c  \\ 
	\end{array}\right).  $$  
Observe that $\widetilde{\Ob}_{d,m}$ may be realized as a semi-direct product $\Ob_{d,m}^+ \rtimes \un{\mu}_2$ 
by means of the section $x \mapsto \diag(1,x) $ of the quotient map $\widetilde{\Ob}_{d,m} \stackrel{\det}{\to} \un{\mu}_2$.  
By \cite[Lemma~2.6.3]{Gil}, this implies the claimed decomposition 
\begin{equation} 
H^1_\fl(\Z,\Ob_{d,m}) = H^1_\fl(\Z,\widetilde{\Ob}_{d,m}) = H^1_\fl(\Z,\un{\textbf{O}}^+_{(1,0,c_{d,m})})/\un{\mu}_2(\Z) \coprod{H^1_\fl(\Z,\un{\textbf{O}}^+_{(1,0,-c_{d,m})})/\un{\mu}_2(\Z)}, 
\end{equation}
where the quotients are taken modulo the equivalence relation given by the action of $\un{\mu}_2(\Z)$ on each group as above.  
Indeed, the twisted form $(1,0,c_{d,m})$ of discriminant $-4c_{d,m}$ corresponds to the non-trivial $\un{\mu}_2$-torsor represented by $\{t^2=-1\}$, 
which splits over $\Z[\sqrt{-1}]$ and is represented by 
$$ 
\left( \begin{array}{cc}
    1 & 0  \\ 
    0 & \sqrt{-1}  \\ 
	\end{array}\right)  
	\left( \begin{array}{cc}
    1 &  0  \\ 
    0 & -c_{d,m}  \\ 
	\end{array}\right) 
	\left( \begin{array}{cc}
    1 & 0  \\ 
    0 & \sqrt{-1}  \\ 
	\end{array}\right)^t
	=
	\left( \begin{array}{cc}
    1 & 0  \\ 
    0 & c_{d,m}  \\ 
	\end{array}\right).\qedhere
$$ 
\end{proof}

\begin{cor}
Each of the groups $H^1_\fl(\Z,\Ob_{d,m}^+)$ and $H^1_\fl(\Z,\un{\mathbf{O}}^+_{(1,0,c_{d,m})})$ is entirely embedded in $H^1_\fl(\Z,\Ob_{d,m})$
if and only if it satisfies one (hence all) of the conditions of Lemma~\ref{TFAE}. 
\end{cor}

\begin{remark} \label{N'}
Although the groups $\widetilde{\Ob}_{d,m}$ and $\widetilde{\Ob}_{(1,0,c_{d,m})}$ are not isomorphic, Lemma~\ref{H1 2,3 mod 4} provides the same decomposition of $H^1_\fl(\Z,\un{\mathbf{O}}_{d,m})$ as of $H^1_\fl(\Z,\un{\mathbf{O}}_{(1,0,c_{d,m})})$.  Hence these two pointed sets are in bijection with each other.  Observe that
$$ \widetilde{\Ob}_{(1,0,c_{d,m})} =
\begin{cases}
\widetilde{\Ob}_{-d, m/2} &: d \equiv 1 \, \mathrm{mod} \, 4, \, m \, \mathrm{even} \\
\widetilde{\Ob}_{-d,m} &: d \equiv 2 \, \mathrm{mod} \, 4 \\
\widetilde{\Ob}_{-d, 2m} &: d \equiv 3 \, \mathrm{mod} \, 4.
\end{cases}
$$ 
\end{remark}


\begin{example} 
In this and the subsequent examples in this section, we set $\N^\prime_d = \N_{d, 2}$ for brevity.
The set $H^1_\fl(\Z,\N_{11})$ contains $2h_{11}=2$ classes $\{[\pm(1,0,-11)]\}$; see Lemma~\ref{H1 Nd and Pic+} below. Each of these classes is a separate $\un{\mu}_2(\Z)$-orbit.
However, $H^1_\fl(\Z,\N^\prime_{-11})$ contains six classes by~\cite[p.20]{Bue}; see the table in Example~\ref{ex:table} below.  
Precisely, we have 
$H^1_\fl(\Z,\N^\prime_{-11}) = \{[\pm(1,0,11)],[\pm(3,\pm 2,4)]\}$.  
The pairs 
$(3,\pm 2,4)$ and $(-3, \pm 2, -4)$ coincide in {$H^1_\fl(\Z,\un{\mathbf{O}}_{11})$. 
Thus $|H^1_\fl(\Z,\un{\mathbf{O}}_{11})|=2+4=6$}.    
\end{example}

Our next aim is to use Lemma~\ref{H1 2,3 mod 4} to study $\Pic \Z[\sqrt{d}]$ even in cases where $\Z[\sqrt{d}]$ is not the maximal order of a number field and thus need not be a Dedekind domain.  As a preliminary, we record the following exercise in algebraic number theory.

\begin{lem} \label{lem:ideals.2}
Let $d$ be any integer.  The ring $\Z [\sqrt{d}]$ has a unique prime ideal containing $2$.
\end{lem}
\begin{proof}
The case $d \in \{ 0, 1 \}$ is obvious, so we assume that it does not hold.  We may assume without loss of generality that $d$ is square-free.  If $d \equiv 2,3 \,(\text{mod}\, 4)$, then $\Z[\sqrt{d}] = \mathcal{O}_d$ is a Dedekind domain and $2$ ramifies in $\Q ( \sqrt{d})$, so that $2 \mathcal{O}_d = \mathfrak{p}^2$, where $\mathfrak{p}$ is the unique prime ideal of $\mathcal{O}_d$ dividing $(2)$.  Now suppose that $d \equiv 1 \, (\text{mod}\, 4)$.  Then $\mathcal{O}_d / \Z[\sqrt{d}]$ is an integral extension of rings, so by~\cite[Theorem~9.3]{Mat} any prime ideal of $\Z[\sqrt{d}]$ has the form $\Z[\sqrt{d}] \cap \mathfrak{p}$, where $\mathfrak{p}$ is a prime ideal of $\mathcal{O}_d$.  If $d \equiv 5 \, (\text{mod}\, 8)$, then $2$ is inert in $\Q(\sqrt{d})$.  Thus $2 \mathcal{O}_d$ is prime and is the unique prime ideal of $\mathcal{O}_d$ containing $2$; this implies our claim by the previous observation.  If $d \equiv 1 \, (\text{mod}\, 8)$, then $2 \mathcal{O}_d = \fp_1 \fp_2 = \fp_1 \cap \fp_2$ for distinct prime ideals $\fp_1$ and $\fp_2$ of $\mathcal{O}_d$.  Hence $\Z[\sqrt{d}] \cap \fp_1$ and $\Z[\sqrt{d}] \cap \fp_2$ each contain $I = \Z[\sqrt{d}] \cap 2 \mathcal{O}_d = \{ a + b\sqrt{d}: a \equiv b \, (\text{mod}\, 2) \}$.  Since $I$ has index $2$ in the ring $\Z[\sqrt{d}]$ and thus is a maximal ideal, it is the unique prime ideal of $\Z[\sqrt{d}]$ containing $2$.
\end{proof}

\begin{prop} \label{order Picard group}
Let $d \equiv 3 \, (\mathrm{mod}~4)$.  Set $\eta(d) = 1$ if $d \equiv 3 \, (\mathrm{mod}~8)$ and one of the following two conditions holds:
\begin{itemize}
\item $d > 3$
\item $d < 0$ and $\mathcal{O}_{-d}^\times \subset \Z[\sqrt{-d}]$.
\end{itemize}
Otherwise, set $\eta(d) = 0$.  Then $|\Pic \Z[\sqrt{-d}]| = 3^{\eta(d)}  h_{-d}$.
\end{prop}

\begin{proof}
Observe that $-d \equiv 1 \, \mathrm{mod} \, 4$, and thus $\Z[\sqrt{-d}] = \CO_{-d,2}$.  We write $\Om$ for $\Om_{-d,2} = \{ 1, \sqrt{-d} \}$.  

The claim is clear if $d = -1$, so suppose that $d \equiv 3 \, (\text{mod} \, 4)$ and $d \neq -1$. 
By Lemma~\ref{N = O+}, the special orthogonal group $\N_{-d,2}$ 
is equal to $\un{\mathbf{O}}_\Om^+$.  
Set $k' = \Q(\sqrt{-d})$.
We relate the Picard groups of $\CO_\Omega = \Z[\sqrt{-d}]$ and $\CO_{-d,1} = \CO_{k'} = \Z[\fc{1+\sqrt{-d}}{2}]$ by studying their localizations. 
For any prime ideal $\fp$ of $\CO_\Omega$, let $\CO_\fp$ be the localization of $\CO_\Omega$ at $\fp$, 
and let $(\CO_{k'})_\fp$ be the integral closure of $\CO_\fp$ in $\CO_{k'}$.  
Then~\cite[Theorem~5.6]{KP} provides an exact sequence of abelian groups 
\begin{equation} \label{Pics} 
1 \to \CO_\Omega^\times \stackrel{\varphi}{\to} 
\CO_{k'}^\times \to \bigoplus_\fp (\CO_{k'})_\fp^\times / \CO_\fp^\times \to \Pic \CO_\Omega \to \Pic \CO_{k'} \to 1. 
\end{equation}
Here the direct sum in the middle runs over the prime ideals of $\CO_\Omega$.
Let $\F$ denote the conductor of $\CO_{k'}/\CO_\Omega$. 
By~\cite[Proposition 6.2]{KP} we have the following isomorphism for any $\fp$: 
\begin{equation} \label{iso localizations ratio}
(\CO_{k'})_\fp^\times / \CO_\fp^\times \cong ((\CO_{k'})_\fp / \F \cdot (\CO_{k'})_\fp)^\times \Big/ (\CO_\fp / \F \CO_\fp)^\times.  
\end{equation} 
Since $\CO_\Omega = \Z + 2\CO_{k'}$, the conductor is $\F = 2\CO_{k'}$. 
It is a maximal ideal of $\CO_\Omega$; 
since localization at any prime commutes with factorization modulo $\F$, we have $(\CO_\fp / \F \CO_\fp)^\times = (\CO_\Omega / \F \CO_\Omega)_\fp^\times =\BF_2^\times=1$. 
Moreover, we see that $(\CO_{k'})_\fp^\times \cong \CO_\fp^\times$ if $2 \not\in \fp$, 
so such places make no contribution to the direct sum in~\eqref{Pics}.    
It remains therefore to compute $(\ov{\CO}_{k'})_{\fq}^\times$ for the unique (by Lemma~\ref{lem:ideals.2}) 
place $\fq$ containing $2$, where $\ov{\CO}_{k'}$ denotes the reduction of $\CO_{k'}$ modulo $\F$.  
Note that $\mathcal{O}_\Omega / \F \simeq \BF_2$ and that $c = \frac{1+d}{4}$ is odd if $d \equiv 3 \,(\text{mod}\, 8)$ and even if $d \equiv 7 \, (\text{mod}\, 8)$. 
If $\bar{c} \in \mathbb{F}_2$ is the image of $c$, then it follows from~\eqref{Ad} that  
\begin{align*} 
(\ov{\CO}_{k'})_\fq^\times \cong (\ov{R}_{-d}(\BF_2))_\fq & =
\left \{ 
(\ov{A}_{-d})_\fq = 
\left( \begin{array}{cc}
    a & \bar{c} b   \\ 
    b & a+b  \\ 
	\end{array}\right) : \, a,b \in \BF_2, \, \det (\ov{A}_{-d}) \neq 0 
\right\} \\ \nonumber
&\cong \left \{ \begin{array}{l l}
    \Z/3  & d \equiv 3 \, (\text{mod}\, 8)  \\ 
    1     & d \equiv 7 \, (\text{mod}\, 8). 
\end{array}\right. 
\end{align*} 
This holds also when $d = -1$.
We deduce from~\eqref{Pics} that if $d \equiv 7 \, (\text{mod}\, 8)$, then $\Pic \mathcal{O}_\Omega \simeq \Pic \mathcal{O}_{-d}$.
If $d \equiv 3 \, (\text{mod} \, 8)$, then~\eqref{Pics} implies that
\begin{equation} \label{Pic ratio}
\fc{|\Pic(\CO_\Omega)|}{h_{-d}} = \fc{3}{|\cok(\p)|}.   
\end{equation}  
If $d \neq 3$, the unit groups $\CO_\Omega^\times$ and $\CO_{k'}^\times$ contain the same roots of unity. 
Moreover, if $d>0$, these groups have no free part, hence $| \cok(\p)|=1$.  If $d = 3$, then clearly $| \cok(\p)|=3$.
If $d<0$, then the free parts of both $\CO_\Omega^\times$ and $\CO_{k'}^\times$ have rank $1$, and $\cok(\p)=[\langle \ve \rangle : \langle\ve^m \rangle]=m$, 
where $\ve$ is a fundamental unit of $k'$ and $\ve^m$ is a generator of $\CO_\Omega^\times$. 
Note that $m | 3$ by~\eqref{Pics}, so that $m = 3$ or $m = 1$.  Both cases do arise.  The case $m = 1$ occurs, for instance, when $d = -37$ and $d = -101$; see sequence A108160 in the {\emph {On-Line Encyclopedia of Integer Sequences}}.  Hence, recalling the definition of $\eta(d)$ from the statement of this claim, we may rewrite~\eqref{Pic ratio} as
\begin{equation*} \label{equ:pic.ratio.new}
\fc{|\Pic(\CO_\Omega)|}{h_{-d}}  = 3^{\eta(d)}. \qedhere
\end{equation*}
\end{proof}

\begin{remark} \label{rmk:norm.equality} 
If the norm map $\Nr$ attains the value $-1$ for an element of $\CO_\Omega^\times$, it does so for a generator of its free part. 
As $m$ is odd, this implies that $\Nr(\CO_{k'}^\times) = \Nr(\CO_\Omega^\times)$. 
Recalling that $\N_d^\prime = \N_{d,2}$, we conclude by~\eqref{N and Pic(Ok)} 
and Corollary~\ref{order Picard group} that 
\begin{equation} 
\fc{|H^1_\fl(\Z,\N_{-d}')|}{|H^1_\fl(\Z,\N_{-d})|} = \fc{|\Pic \Z[\sqrt{-d}]|}{h_{-d}} = 3^{\eta(d)}. 
\end{equation} 
\end{remark}
\bk

\begin{example} \label{ex:table}
We tabulate the following data from~\cite{Bue}: see page~19 for the second and fourth columns and page~20 for the third, 
noting that, as the forms obtained are definite, the number of total classes is twice the number of positive classes 
by Proposition~\ref{Gauss Z[root(d)]} below.  

\bk
\begin{center}
 \begin{tabular}{|c | c | c | c| c| c|} 
 \hline
 $0<d \equiv 3 \, (\text{mod} 4)$ & $h_{-d}$ & $|H^1_\fl(\Z,\N_{-d}')|$ & $|H^1_\fl(\Z,\N_{-d})|$ & $c_{-d,1}=\fc{1+d}{4}$ 
    \\ [0.5ex] \hline\hline
  3 & 1 &  2 & 2 & 1 \\
 \hline
  7 & 1 &  2 & 2 & 2 \\
 \hline
 11 & 1 &  6 & 2 & 3 \\
 \hline
 15 & 2 &  4 & 4 & 4 \\
 \hline
 19 & 1 &  6 & 2 & 5 \\
 \hline
 23 & 3 &  6 & 6 & 6 \\
 \hline
\end{tabular}
\end{center}
\end{example}

\begin{example} 
Let $d=-5$.  Then $(\ov{\CO}_{k'})_{(2)}^\times / \ov{\CO}_{(2)}^\times \cong \Z/3$ by the argument preceding~\eqref{Pic ratio}.  
A generator of the free part of $\CO_{k'}^\times$ is $\ve=\om=\frac{1 + \sqrt{5}}{2}$. 
Let $\Om = \{1,\sqrt{5}\}$. 
The embedding $\p:\CO_\Omega \to \CO_{k'}$ induces the embedding of $\N_{\Om}(\Z)$ in $\N_5(\Z)$ given by the integral matrix realization of~\eqref{Ad}, 
namely the {group homomorphism}
$$
\p:\left( \begin{array}{cc}
    x & 5y  \\ 
    y & x  \\ 
	\end{array}\right) 
	\mapsto 
	\left( \begin{array}{cc}
    x-y  & 2y   \\ 
    2y & x+y  \\ 
	\end{array}\right). 
$$
Since $\mathrm{Nr}(\varepsilon) = -1$, a generator of the free part of $\N_5(\Z)$ is    
$u = \left( \begin{array}{cc}
       1 & 1  \\ 
       1 & 2  \\ \end{array}\right)$, corresponding to $\ve^2 = 1 + \om$, 
while a generator of $\N_5'(\Z)$ is 
$z = \left( \begin{array}{cc}
       9 & 20  \\ 
       4 & 9  \\ \end{array}\right)$, corresponding to $\ve^6 = 9 + 4 \sqrt{5}$.    
Hence 
$$ \p(z) = \left( \begin{array}{cc}
        5 & 8  \\ 
        8 & 13 \\ \end{array}\right) = u^3,    
$$ 
so that $|\cok(\p)|=3$. This means that $\eta(d) = 0$ in Proposition~\ref{order Picard group}, and $|\Pic(\Z[\sqrt{5}])| = h_5$. 
\end{example}  

Recall from~\eqref{equ:def.cl.plus} that the set $\cl^+(n)$ classifies quadratic forms of discriminant $n$ up to proper isometry.  The next lemma relates $\cl^+(n)$ to the flat cohomology of special orthogonal groups.

\begin{lem} \label{discriminant classification}
Let $d$ be any integer that is not a perfect square and not of the form $d = -3 \cdot 4^m$ for any $m \in \mathbb{N}_0$, and set $q(x,y) = x^2 - dy^2$. Then $\cl^+(4d) = H^1_\fl(\Z,\OV^+)$. 
\end{lem}
\begin{proof} 
Note that $\Delta(q) = 4d$ need not be a fundamental discriminant.
The equivalence relations in $\cl^+(4d)$ and in $H^1_\fl(\Z,\OV^+)$ are defined similarly, 
so we only need to show that the two sets of representatives coincide. 
Indeed, those in $H^1_\fl(\Z,\OV^+)$ are obtained by (local) proper isometries of $q$, resulting in the same discriminant $\Dl(q)$.   
Conversely, we claim that any quadratic $\Z$-form $q'$ of rank $2$ with $\Dl(q')=\Dl(q)$ 
is diagonalizable over ${\Z}_p$ for any prime $p$, thus is locally isomorphic to $q$ in the fppf topology.  
This is true for all odd primes $p$ by~\cite[Theorem~8.3.1]{Cas}.  
If $p = 2$, then by the explicit form of the Jordan Decomposition Theorem~\cite[Lemma~8.4.1]{Cas} $q^\prime$ is isomorphic over $\Z_2$ to a direct product 
of forms of rank at most $2$, where the possible components of rank $2$ are of the form $2^e xy$ or $2^e (x^2 + xy + y^2)$ for $e \in \mathbb{N}_0$; note that~\cite{Cas} uses the classical definition of integral forms, requiring that $b$ be even.
None of these has discriminant $4d$ for an integer $d$ satisfying our hypotheses.
Hence $q^\prime$ is diagonalizable over $\Z_2$.
Moreover, by Lemma~\ref{H1 2,3 mod 4}  a local isometry between $q$ and $q'$      
can be taken to be proper, as they share the same discriminant.          
So any $\Z$-form with discriminant $\Dl(q)$ is locally properly isomorphic to $q$ in the flat topology. 
This completes the proof.       
\end{proof}

\begin{remark} \label{rem:bad.discriminants}
We justify the exclusion of discriminants of the form $-3 \cdot 4^m$, for $m \geq 1$, in the hypotheses of Lemma~\ref{discriminant classification} by noting that the lemma is false in those cases.  Indeed, the form $(2^m, 2^{m-1}, 2^m)$ has discriminant $-3 \cdot 4^m$, but one readily checks that it is not isometric to $(1, 0, 3 \cdot 4^{m-1})$ over any finite flat extension of $\Z_2$.  Thus the class $[(2^m, 2^{m-1}, 2^m)]$ appears in $\cl^+(-3 \cdot 4^m)$ but not in $H^1_\fl(\Z, \un{\mathbf{O}}_{(1,0,3 \cdot 4^{m-1})}^+)$.  Indeed, $\Pic \Z[\sqrt{-3}]$ is trivial; see, for instance,~\cite[Exercise~7.9]{Cox}.  It follows from~\eqref{N and Pic(Ok)} and Lemma~\ref{N = O+} that $|H^1_\fl (\Z, \un{\mathbf{O}}_{(1,0,3)}^+)| = 2$; alternatively, see the data of Example~\ref{ex:table}.  Thus $H^1_\fl (\Z, \un{\mathbf{O}}_{(1,0,3)}^+) = \{ [\pm(1,0,3)] \}$.
\end{remark}

\subsection{Applications to the classification of binary quadratic forms}
Having studied Cases (I) and (II)
separately,
we gather together our results.  First we determine the structure of $\cl(D)$ for many integers $D$.

\begin{prop} \label{pro:cl.case.i}
Let $D \neq -3$ be an integer such that $D \equiv 1 \, \mathrm{mod} \, 4$ and $D$ is not a perfect square.  Set $q(x,y) = x^2 + xy + ((1 - D)/4) y^2$.  Then $\cl(D) = \cl^+(D) = H^1_\fl(\Z,\OV^+)$.
\end{prop}
\begin{proof}
The same argument as in the proof of Lemma~\ref{discriminant classification} shows that $\cl^+(D) = H^1_\fl(\Z,\OV^+)$.  Write $D = dm^2$, where $d \equiv 1 \, \mathrm{mod} \, 4$ is square-free and $m$ is odd.  Then $q = q_{d,m}$.  Moreover, by Lemma~\ref{N = O+} and Proposition~\ref{H1 1 mod 4} we have $H^1_\fl(\Z, \OV^+) = H^1_\fl(\Z, \un{N}_{d,m}) = H^1_\fl(\Z, \Ob_{d,m})$.  Thus any quadratic form that is improperly isomorphic to $q$ is also properly isomorphic to $q$, hence $\cl(D) = \cl^+(D)$.
\end{proof}

\begin{prop} \label{pro:cl.case.ii}
Let $D$ be any integer that is not a perfect square and not of the form $D = -3 \cdot 4^\ell$ for any $\ell \in \mathbb{N}_0$.  Set $q(x,y) = x^2 - Dy^2$.  Then $\cl(4D) = H^1_\fl(\Z, \OV^+)/ \un{\mu}_2(\Z)$, where the non-trivial element of $\un{\mu}_2(\Z)$ maps $[(a,b,c)]$ to $[(a,-b,c)]$.
\end{prop}
\begin{proof}
We have $\cl^+(4D) = H^1_\fl(\Z, \OV^+)$ by Lemma~\ref{discriminant classification}.  The description of the structure of $H^1_\fl(\Z, \OV)$ provided by Lemma~\ref{H1 2,3 mod 4} shows that the only proper isomorphism classes of forms {\emph{of discriminant $4D$}} that are improperly isomorphic to each other are those in the same orbit of the $\un{\mu}_2(\Z)$-action.
\end{proof}
 
Recall that $k = \Q(\sqrt{d})$.
As noted in the introduction, 
Gauss proved in his {\emph{Disquisitiones Arithmeticae}} \cite{Gau} 
that the elements of $\cl^+(\Dl(q_d))$, namely proper isomorphism classes of forms of discriminant $\Dl_k$, 
are parametrized by $\mathrm{Pic}^+(\CO_k)$.  See, for instance, \cite[Theorem~58]{FT} for an exposition of this result. 
If $d<0$, his classification treats only the positive definite forms, namely those for which $a, c > 0$.  
The following claim completes the proper classification.

\begin{thm} \label{H1 Nd and Pic+}
If $d \not\in \{ 0, 1 \}$ is a square-free integer, then 
$$ \cl^+(\Dl(q_d)) = H^1_\fl(\Z,\Od^+) \cong \{\pm 1\}^{\mu(d)} \times \mathrm{Pic}^+(\CO_k), \ \text{where} \  
\mu(d) =
\left \{ 
\begin{array}{l l}
1 & d<0 \\ 
0 & d>0.
\end{array}\right.$$
\end{thm}

\begin{proof} 
By Lemma~\ref{N = O+} we have 
$H^1_\fl(\Z,\N_d) = H^1_\fl(\Z, \Od^+)$, and the latter 
properly classifies the integral quadratic forms 
that are locally isomorphic to $q_d$ for the flat topology, 
thus of discriminant $\Delta_k$.  
So if $d>0$, then $H^1_\fl(\Z,\N_d)$ injects into $\cl^+(\Dl_k) = \cl^+(\Dl(q_d))$,
which is in bijection with $\mathrm{Pic}^+(\CO_k)$ by the classical theorem of Gauss mentioned above.
Since $H^1_\fl(\Z,\N_d)$ and $\mathrm{Pic}^+(\CO_k)$ have the same cardinality by~\eqref{mu quadratic}, we have obtained a natural bijection between them.

If $d<0$, however, then $\mathrm{Pic}^+(\CO_k)$ classifies only the {positive} definite forms; 
see the proof of~\cite[Theorem~58]{FT}.  
The subset $H^1_\fl(\Z,\N_d)^+$ of classes of positive forms injects into $\mathrm{Pic}^+(\CO_k) = \Pic \CO_k$ as above. 
If $[q] \in H^1_\fl(\Z,\N_d)$, then the isometry $\diag(\sqrt{-1},\sqrt{-1})$ shows that $[-q]$ belongs to $H^1_\fl(\Z,\Od)$.  
Thus if $d \equiv 1 \, (\text{mod} \, 4)$, then $[-q] \in H^1_\fl (\Z,\N_d)$ by Proposition~\ref{H1 1 mod 4}.
This remains true also in the case $d \equiv 2,3 \, (\text{mod} \, 4)$ by the decomposition of Lemma~\ref{H1 2,3 mod 4}, since $\disc(-q)=\Delta_k$.     
Furthermore, since $q$ realizes only non-negative values and $-q$ realizes non-positive values, 
the two forms $q$ and $-q$ cannot be $\Z$-equivalent. 
Since every definite form is positive or negative, we have $\{\pm 1\} \times H^1_\fl(\Z,\N_d)^+ = H^1_\fl(\Z,\N_d)$,
and we have just shown that this
injects into {$\{\pm 1\} \times \mathrm{Pic}^+(\CO_k) = \cl^+(\Dl(q_d))$}. 
Again by~\eqref{mu quadratic}, these sets have the same cardinality, so our injection is a bijection.  
\end{proof}

The following statement generalizes the result of Gauss mentioned above  
to cases in which $4d$ is not a fundamental discriminant, and, when $d < 0$, to forms that need not be positive definite.  It is also classical; see, for instance,~\cite[Exercise~7.23]{Cox}.  We prove it by modern methods, at the price of ruling out discriminants of the form $-3 \cdot 4^\ell$.  The theorem remains true in this case, but Remark~\ref{rem:bad.discriminants} shows that our cohomological argument must be modified in order to treat it.
If $D$ is an integer that is not a perfect square, then define
\begin{align} \label{def:epsilon.tilde}
\tilde{\ve}(D) &=
\left \{ \begin{array}{l l}
0 & D > 0 \ \text{and} \ \mathrm{Nr}(\Z[\sqrt{D}]^\times) = \{ \pm 1 \} \\
1 & \text{otherwise.} \\
\end{array}\right. 
\end{align} 

\begin{thm} \label{Gauss Z[root(d)]}  
For any integer $D$ that is not a perfect square and not of the form $D = -3 \cdot 4^\ell$ for any $\ell \in \mathbb{N}_0$,
there is a bijection  
$$ \cl^+(4D) \cong \{\pm 1\}^{\tilde{\ve}(D)} \times \mathrm{Pic}(\Z[\sqrt{D}]). $$
Moreover, if $d \not\in \{ 0,1, -3 \}$ is square-free, then
$ \cl^+(4d) \cong \{\pm 1\}^{\ve(d)} \times \mathrm{Pic}(\Z[\sqrt{d}]).$  
\end{thm}

\begin{proof}
Since $D$ is not a perfect square, it may be written uniquely in the form $D = d (m^\prime)^2$, where $d \neq 0,1$ is square-free and $m^\prime \in \mathbb{N}$.  Set $m = 2m^\prime$ if $d \equiv 1 \, \mathrm{mod} \, 4$ and $m = m^\prime$ otherwise.  Observe that $\Z[\sqrt{D}] = \CO_{d,m}$ and $q_{d,m} = (1,0,-D)$.  
Consider the exact sequence~\eqref{N and Pic(Ok)} for the basis $\Omega_{d,m}$: 
$$
1 \to \{\pm 1\}/\Nr(\CO_{d,m}^\times) \to H^1_\fl(\Z,\N_{d,m}) \to \Pic \Z[\sqrt{D}] \to 1.   
$$
The group $H^1_\fl(\Z,\N_{d,m})$ is equal to $H^1_\fl(\Z,\textbf{O}_{d,m}^+)$ by Lemma~\ref{N = O+}, and hence
is identified with $\cl^+(4D)$ by Lemma~\ref{discriminant classification}.  The first part of our claim is immediate.  For the second part, there is nothing further to prove if $d \equiv 2,3 \, \mathrm{mod} \, 4$; in this case, $\ve(d) = \tilde{\ve}(d)$ by definition.  If $d \equiv 1 \, \mathrm{mod} \, 4$ is square-free, then 
we observed in Remark~\ref{rmk:norm.equality} that $\Nr(\Z[\sqrt{D}]^\times) = \Nr(\CO_d^\times)$, 
so again $\ve(d) = \tilde{\ve}(d)$.  This concludes the proof. 
\end{proof}

\begin{prop} \label{H1 cardinality}
For a square-free integer $d \not\in \{ 0, 1 \}$   
let $m_d$ be the number of pairs $[(a,\pm b,c)]$ which are distinct in $\cl^{+}(\Dl(q_d))$,  
and let $l_d$ be the number of such pairs in $\cl^{+}(-\Dl(q_d))$.  
Let $h_d^+$ be the narrow class number of $\Q(\sqrt{d})$. 
Then
\begin{align*}
|H^1_\fl(\Z,\Od)| =
\left \{ \begin{array}{l l l}
2^{\mu(d)} h^+_d                                             &  d \equiv 1 \, (\text{mod} \, 4) \\ \\
2^{\mu(d)} h^+_d + 2^{\mu(-d)}          h^+_{-d} -m_d - l_d  &  d \equiv 2 \, (\text{mod} \, 4) \\ \\ 
2^{\mu(d)} h^+_d + 2^{\mu(-d)} \cdot 3^{\eta(d)} h_{-d}^+ - m_d - l_d  &  d \equiv 3 \, (\text{mod} \, 4). \\   
\end{array}\right. 
\end{align*}   
\end{prop}

\begin{proof} 
If $d \equiv 1 \, (\text{mod} \, 4)$ then
$|H^1_\fl(\Z,\Od)| = |H^1_\fl(\Z,\N_d)| 
                   = h^+_d \cdot 2^{\mu(d)}$, where the first equality is Proposition~\ref{H1 1 mod 4} and the second comes from~\eqref{mu quadratic}.
Otherwise, use Lemma~\ref{H1 2,3 mod 4} 
and notice that if $d \equiv 2 \, (\text{mod} \, 4)$, then $-d \equiv 2 \, (\text{mod} \, 4)$ as well, so $\N_{-d}'=\N_{-d}$.  
On the other hand, if $d \equiv 3 \, (\text{mod} \, 4)$, then $-d \equiv 1 \, (\text{mod} \, 4)$, and the claim follows.
\end{proof}

\begin{remark}
Let $d \equiv 2 \, \mathrm{mod} \, 4$ be square-free.  Then $|H^1_\fl(\Z,\Od)| = |H^1_\fl(\Z,\un{\mathbf{O}}_{-d})|$ by Lemma~\ref{H1 2,3 mod 4}.  However, we see from Proposition~\ref{H1 cardinality} that this observation does not readily imply any relation between the class numbers $h_d$ and $h_{-d}$.  Indeed, the class numbers of real and imaginary quadratic fields behave very differently.  For instance, it is well known that only nine imaginary quadratic fields have class number one, but the Cohen-Lenstra heuristics suggest that this should be true of infinitely many real quadratic fields.
\end{remark}

\subsection{On the principal genus theorem} 
Theorem~\ref{H1 Nd and Pic+} also implies another classical result of Gauss: the principal genus theorem.  Recall that $k = \Q(\sqrt{d})$, with $d \not\in \{ 0,1 \}$ square-free.  Then for any binary quadratic form of discriminant $\Delta_k$, the composition of $q$ with itself belongs to the principal genus, namely the genus of the norm form $q_d$.  Under the identification of Lemma~\ref{discriminant classification}, composition of quadratic forms corresponds to multiplication in the abelian group $H^1_\fl(\Z,\Od^+)$; 
we refer the reader to the discussion in~\cite[Section~V.7.3]{Knus} for details.

\begin{cor} \label{exponent 2}
Let $d \not\in \{ 0, 1 \}$ be square-free and $m \geq 0$.  For any class $[q] \in H^1_\fl(\Z,\Ob_{d,m}^+)$, the class $[q \otimes q]$ belongs to the principal genus. 
\end{cor}

\begin{proof}
To alleviate the notation, we write $\un{N}$ for $\un{N}_{d,m}$.
We have seen in~\eqref{embedding of local N} 
that {$H^1_\fl(\Z_p,\N_p)$} injects into $H^1(\Q_p,N_p)$ for any $p$. 
As a result, by~\eqref{Nis sequence} and Lemma \ref{N = O+} we obtain that 
\begin{equation*} 
\text{Cl}_\iy(\Ob_{d,m}^+) = \text{Cl}_\iy(\N)  = \ker[H^1_\fl(\BZ,\N) \to H^1(\Q,N)],  
\end{equation*}
showing that $\text{Cl}_\iy(\Ob_{d,m}^+)$ is the principal genus of $q$. 
Moreover, the quotient $H^1_\fl(\Z,\N) / \text{Cl}_\iy(\N) = \Sh^1_{S_r \cup \{\iy\}}(\Q,N)$ 
has exponent $2$ by Proposition~\ref{h(N)} and Remark~\ref{B}.  
Recall that $\N$ is commutative, so that $H^1_\fl(\Z,\N)$ is an abelian group.  Thus for any class $[q] \in H^1_\fl(\Z,\N_d)$, 
the class of the tensor product $q \otimes q$ lies in $\text{Cl}_\iy(\N) = \text{Cl}_\iy(\Ob_{d,m}^+)$. 
\end{proof}


\begin{remark} 
Proposition~\ref{h(N)} shows that $\text{Cl}_\iy(\N_{d,m})$ embeds as a subgroup in $H^1_\fl(\Z,\N_{d,m})$.  
If $d \neq -3$, then the latter group is a disjoint union of classes of integral quadratic binary forms of discriminant $\Delta(q_{d,m})$ of all genera.    
This embedding holds for any twisted form of $q_{d,m}$,   
hence the quotient $H^1_\fl(\Z,\N_{d,m}) / \text{Cl}_\iy(\N_{d,m}) \simeq \Sh^1_{S_r \cup \{\iy\}}(\Q,N_{d,m})$ 
is in bijection with the set of proper genera of $q_{d,m}$.     
Thus there are $2^{|S_r|-1}$ such proper genera, as was initially proved by Gauss; see also \cite[\S 5, Example~2]{Ono} and \cite[Corollary~16]{Wat}.  
\end{remark}\bk

\begin{appendices}
\section{Appendix: Some explicit presentations of group schemes} \label{subsec:explicit}
In this section we will write down equations cutting out the algebraic groups $\widetilde{\Ob}_{d,m}$ and $\Ob_{d,m}^+$ and use them to provide an explicit proof of Proposition~\ref{prop:quotient}, which describes their quotient.

\begin{lem} \label{pro:od.tilde.explicit}
Let $d \neq 0,1$ be square-free and $m \in \mathbb{N}$, and define the ideal $\mathcal{I}_{d,m} \subset \Z[\alpha, \beta, \gamma, \delta, t]$ of a polynomial ring in five variables as follows.  In Case (I), set
$$
\mathcal{I}_{d,m} = 
(\alpha^2 + \alpha \beta + c_{d,m} \beta^2 - 1, 2 \alpha \gamma + \alpha \delta + \beta \gamma + 2c_{d,m} \beta \delta - 1, \gamma^2 + \gamma \delta + c_{d,m} \delta^2 - c_{d,m}, (\alpha \delta - \beta \gamma)t - 1, u^2 - u),$$
where $c_{d,m}$ was defined in~\eqref{equ:c} and $u = 1 - \alpha \gamma - \beta \gamma -c_{d,m} \beta \delta$.

In Case (II), define
$
\mathcal{I}_{d,m} = (\alpha^2 - c_{d,m} \beta^2 - 1, \alpha \gamma - c_{d,m} \beta \delta, \gamma^2 - c_{d,m} \delta^2 + c_{d,m}, (\alpha \delta - \beta \gamma)t - 1, t^2 - 1).$

Then $\widetilde{\Ob}_{d,m} = \mathrm{Spec} \, \Z[\alpha, \beta, \gamma, \delta, t] / \mathcal{I}_{d,m}$.
\end{lem}
\begin{proof}
We first treat Case (II), in which $q_{d,m}(x,y) = x^2 - cy^2$; we write $c$ for $c_{d,m}$ for brevity.  Consider the matrix
$$ A = \left( \begin{array}{cc} \alpha & \beta \\ \gamma & \delta \end{array} \right).$$
The identity $q_{d,m} \circ A = q_{d,m}$ amounts to
\begin{equation*}
q_{d,m} = (\alpha x + \gamma y)^2 - c (\beta x + \delta y)^2 = (\alpha^2 - c \beta^2) x^2 + (2 \alpha \gamma - 2c \beta \delta)xy + (\gamma^2 - c \delta^2) y^2.
\end{equation*}
Equating the coefficients of $x^2$, $xy$, and $y^2$, and noting that $A$ must be invertible, we obtain
\begin{equation*}
\Ob_{d,m} = \mathrm{Spec} \, \Z [\alpha, \beta, \gamma, \delta, t] / (\alpha^2 - c \beta^2 - 1, \alpha \gamma - c \beta \delta, \gamma^2 - c \delta^2 + c, (\alpha \delta - \beta \gamma)t - 1).
\end{equation*}
Observe that the identity $c((\alpha \delta - \beta \gamma)^2 - 1) = 0$ is satisfied in $\Ob_{d,m}$.  This can be seen by taking determinants in the matrix equation $A B_{q_{d,m}} A^t = B_{q_{d,m}}$ corresponding to the condition $q_{d,m} \circ A = q_{d,m}$; alternatively, we leave it as an exercise for the reader to deduce this identity from the equations cutting out $\Ob_{d,m}$.  Therefore, $(\alpha \delta - \beta \gamma)^2 - 1 = 0$ (or, equivalently, $t^2 - 1 = 0$), in $\widetilde{\Ob}_{d,m}$.  One now checks explicitly that the ring
$\Z [\alpha, \beta, \gamma, \delta, t] / (\alpha^2 - c \beta^2 - 1, \alpha \gamma - c \beta \delta, \gamma^2 - c \delta^2 + c, (\alpha \delta - \beta \gamma)t - 1, t^2 - 1)$
is torsion-free, and this implies the claim in Case (II).

Now suppose that we are in Case (I), i.e.~that $d \equiv 1 \, \mathrm{mod} \, 4$ and $m$ is odd.  Then $q_{d,m} = x^2 + xy + c y^2$.  In this case, the condition $q_{d,m} \circ A = q_{d,m}$ gives
\begin{align*}
q_{d,m} = & (\alpha x + \gamma y)^2 + (\alpha x + \gamma y)(\beta x + \delta y) + c (\beta x + \delta y)^2 = \\ & (\alpha^2 + \alpha \beta + c \beta^2)x^2 + (2 \alpha \gamma + \alpha \delta + \beta \gamma + 2c \beta \delta)xy + (\gamma^2 + \gamma \delta + c \delta^2) y^2.
\end{align*}
Hence we obtain
\begin{equation*}
\Ob_{d,m} = \mathrm{Spec} \, \Z[\alpha, \beta, \gamma, \delta, t] / (\alpha^2 + \alpha \beta + c \beta^2 - 1, 2 \alpha \gamma + \alpha \delta + \beta \gamma + 2c \beta \delta - 1, \gamma^2 + \gamma \delta + c \delta^2 - c, (\alpha \delta - \beta \gamma)t - 1).
\end{equation*}
By similar considerations to the previous case, the identity $c((\alpha \delta - \beta \gamma)^2  - 1) = 0$ holds in ${\Ob}_{d,m}$.  Set $u = 1 - \alpha \gamma - \beta \gamma - c \beta \delta$.  Then one of the equations cutting out $\Ob_{d,m}$ may be written as $2u = \alpha \delta - \beta \gamma + 1$.  Hence $c((2u - 1)^2 - 1) = 4c(u^2 - u) = 0$ in $\Ob_{d,m}$, and thus $u^2 - u = 0$ in $\widetilde{\Ob}_{d,m}$.  
Since the ring $\Z[\alpha, \beta, \gamma, \delta, t]/\mathcal{I}_{d,m}$ is torsion-free, the claim follows.
\end{proof}
\begin{cor} \label{cor:special.explicit}
Let $d \neq 0,1$ be square-free and $m \in \mathbb{N}$.  Maintaining the notation of Lemma~\ref{pro:od.tilde.explicit},  we have
$$
\Ob_{d,m}^+ = \begin{cases}
\mathrm{Spec} \, \Z[\alpha, \beta, \gamma, \delta, t] / (\mathcal{I}_d, u-1) &: \mathrm{Case} \, \mathrm{(I)} \\
\mathrm{Spec} \, \Z[\alpha, \beta, \gamma, \delta, t] / (\mathcal{I}_d, t - 1) &: \mathrm{Case} \, \mathrm{(II)}.
\end{cases}
$$
\end{cor}
\begin{proof}
In all cases, the identity $t = 1$ is immediate from the definition of $\Ob_{d,m}^+$.  In Case (I), this implies that the identity $2u - 1 = 1$ holds in $\Ob_{d,m}^+$, and hence so does $u = 1$.  We leave the verification that these rings are torsion-free to the reader.
\end{proof}
\begin{remark} \label{rmk:explicit.equality}
It is an instructive exercise to verify Lemma~\ref{N = O+} for the bases $\Omega_{d,m}$ by showing that the map $x \mapsto \alpha, y \mapsto \beta$ is an explicit isomorphism of commutative rings between the coordinate ring of $\un{N}_{d,m}$ and that of $\Ob_{d,m}^+$.
\end{remark}

With the previous results in hand, we can give an explicit proof of Proposition~\ref{prop:quotient}.

\begin{repeatprop}
Let $d \neq 0,1$ be square-free and $m \in \mathbb{N}$.  Then 
$$ \widetilde{\Ob}_{d,m} / \Ob_{d,m}^+ \simeq 
\begin{cases}
\un{\Z / 2} &: \mathrm{Case} \, \mathrm{(I)} \\
\un{\mu}_2 &: \mathrm{Case} \, \mathrm{(II)}.
\end{cases}
$$
\end{repeatprop}
\begin{proof}
The claim follows straightforwardly from the explicit presentations obtained in Lemma~\ref{pro:od.tilde.explicit} and Corollary~\ref{cor:special.explicit}.  In the course of the proofs of those statements, we showed that the relation $(\alpha \delta - \beta \gamma)^2 - 1 = 0$ always holds in $\widetilde{\Ob}_{d,m}$.  Hence the determinant induces a map $\det: \widetilde{\Ob}_{d,m} \to \un{\mu_2}$ corresponding to the ring homomorphism 
$ \Z [t] / (t^2 - 1) \to \Z[\alpha, \beta, \gamma, \delta, t] / \mathcal{I}_{d,m}$
sending $t$ to $\alpha \delta - \beta \gamma$.  

Now suppose we are in Case (II).  If $R$ is a $\Z$-algebra, then the sequence of groups
$$1 \to \Ob_{d,m}^+ (R) \to \widetilde{\Ob}_{d,m}(R) \stackrel{\det}{\to} \un{\mu}_2(R) \to 1$$
is exact.  Indeed, the only part of this statement not immediate from Lemma~\ref{pro:od.tilde.explicit} is the surjectivity of the determinant, obtained by observing that if $x \in \un{\mu}_2(R)$, then $\mathrm{diag}(x,1) \in \widetilde{\Ob}_{d,m}(R)$.

In Case (I), on the other hand, the determinant need not be surjective.  To see this, recall from the proof of Lemma~\ref{pro:od.tilde.explicit} that if $A \in \widetilde{\Ob}_{d,m}(R)$, then $\det A = 2u - 1$ for $u = 1 - \alpha \gamma  - \beta \gamma - c_{d,m} \beta \delta \in R$, following our usual notation for matrix elements.  Now consider the $\Z$-algebra $R = \Z [t] / (t^2 - 1)$; it is, in fact, faithfully flat over $\Z$ of finite presentation.  Observe that $t \in \un{\mu}_2(R)$, but there is no $u \in R$ such that $2u - 1 = t$.  Thus the determinant map $\widetilde{\Ob}_{d,m}(R) \to \un{\mu}_2(R)$ is not surjective.

Instead, consider the map $D: \widetilde{\Ob}_{d,m} \to \un{\Z / 2}$ corresponding to the ring homomorphism $\Z[t] / (t^2 - t) \to \Z[\alpha, \beta, \gamma, \delta, t] / \mathcal{I}_{d,m}$ sending $t$ to $1 - \alpha \gamma - \beta \gamma - c_{d,m} \beta \delta$.  This notation reflects that $D$ is the Dickson morphism $D_{q_{d,m}}$; cf.~\cite[(C.2.2)]{Con2}.  We claim that 
$$1 \to \Ob_{d,m}^+ (R) \to \widetilde{\Ob}_{d,m}(R) \stackrel{D}{\to} \un{\Z / 2}(R) \to 1$$ 
is an exact sequence of groups for any $\Z$-algebra $R$.  Indeed, if $x \in \un{\Z / 2}(R)$, then 
$$A_x = \left( \begin{array}{cc} 1 & 0 \\ 1 - x & 2x - 1 \end{array} \right) \in \widetilde{\Ob}_{d,m}(R)$$ 
satisfies $D(A_x) = x$, so $\psi$ is surjective.  It follows from Corollary~\ref{cor:special.explicit} that $\ker D = \Ob_{d,m}^+(R)$.
\end{proof}

\end{appendices}


\begin{thebibliography}{[groups]}

\bibitem[SGA4]{SGA4} M.\ Artin, A.\ Grothendieck, J.-L. Verdier, {\em Th\'eorie des Topos et Cohomologie \'Etale des Sch\'emas} (SGA 4) LNM, Springer, 1972/1973. 

\bibitem[Ana]{Ana} S.\ Anantharaman, {\em Sch\'{e}mas en groupes, espaces homog\`{e}nes et espaces alg\'{e}briques sur une base de dimension $1$}, M\'{e}m. Soc. Math. France {\bf 33}, 1973, 5--79.

\bibitem[Bor]{Bor} A.\ Borel, {\em Some finiteness properties of adele groups over number fields}, Publ. Math. IHES {\bf 16}, 1963, 5--30.


\bibitem[BLR]{BLR} S.\ Bosch, W.\ L\"utkebohmert, M.\ Raynaud, {\em N\'eron Models}, Springer, Berlin, 1990. 

\bibitem[Bou]{Bou} N.\ Bourbaki, {\em \'El\'ements de math\'ematique, Alg\`{e}bre commutative}, Hermann, Paris, 1972. 

\bibitem[Bue]{Bue} D.\ A.\ Buell, {\em Binary Quadratic Forms, Classical Theory and Modern Computations}, Springer-Verlag, 1989. 




\bibitem[Cas]{Cas} J.\ W.\ S.\ Cassels, {\em Rational Quadratic Forms}, Academic Press, London-New York, 1978.

\bibitem[CF]{CF} B.\ Calm\`es, J.\ Fasel, {\em Groupes Classiques}. In ``Autour des sch\'{e}mas en groupes'' vol. II.  Panorames et synth\`{e}ses 46, 1-133.  Paris, Soci\'{e}t\'{e} Math\'{e}matique de France, 2015. 

\bibitem[Con1]{Con1} B.\ Conrad, {\em Math 252. Properties of orthogonal groups}, available at \\ 
{\texttt{http://math.stanford.edu/$\sim$conrad/252Page/handouts/O(q).pdf}}



\bibitem[Con2]{Con2} B.\ Conrad, {\em Reductive group schemes}.  In ``Autour des sch\'{e}mas en groupes'' vol. I. Panorames et synth\`{e}ses 42/43, 93-444. Paris, Soci\'{e}t\'{e} Math\'{e}matique de France, 2014.

\bibitem[Con3]{Con3} B.\ Conrad, {\em Non-split reductive groups over $\mathbb{Z}$}.  In ``Autour des sch\'{e}mas en groupes'' vol. II.  Panorames et synth\`{e}ses 46, 193-253.  Paris, Soci\'{e}t\'{e} Math\'{e}matique de France, 2015. 
.

\bibitem[Cox]{Cox} D.\ A.\ Cox, {\em Primes of the form {$x^2 + ny^2$}}, second edition.  J.\ Wiley and Sons, Hoboken, N.J., 2013.



\bibitem[SGA3]{SGA3} M.\ Demazure, A.\ Grothendieck, {\em S\'eminaire de G\'eom\'etrie Alg\'ebrique du Bois Marie - 1962-64 - Sch\'emas en groupes}, Tome I, R\'e\'edition de SGA3 (2011), P.\ Gille, P.\ Polo. 


\bibitem[FT]{FT} A.\ Frohlich, M.\ J.\ Taylor {\em Algebraic Number Theory}, 
Cambridge Studies in Advanced Mathematics 27. 


\bibitem[Gau]{Gau} C.\ F.\ Gauss, {\em Disquisitiones Arithmeticae}, 1801. 

\bibitem[GP]{GP} P.\ Gille, A.\ Pianzola, {\em Isotriviality and \'etale cohomology of Laurent polynomial rings}, Journal of Pure and Applied Algebra {\bf 212}, 2008, 780--800. 

\bibitem[Gil]{Gil} P.\ Gille, {\em Sur la classification des sch\'emas en groupes semi-simples}.  In ``Autour des sch\'{e}mas en groupes'' vol. III. Panorames et synth\`{e}ses 47, 39-110.  Paris, Soci\'{e}t\'{e} Math\'{e}matique de France, 2015.

\bibitem[Gir]{Gir}   J.\ Giraud, {\em Cohomologie non ab\'elienne}, Grundlehren math. Wiss., Springer-Verlag Berlin Heidelberg New York, 1971.

\bibitem[EGAIV]{EGAIV} A. Grothendieck,  {\em \'El\'ements de g\'eom\'etrie alg\'ebrique 
(r\'edig\'es avec la collaboration de J. Dieudonn\'e):  
IV. \'Etude locale des sch\'emas et des morphismes de sch\'emas, Seconde partie}, 
Publications math\'ematiques de l'I.H.E.S. {\bf 24}, 1965, 5--231.

\bibitem[EGAI]{EGAI} A.\ Grothendieck, J.\ Dieudonn\'{e}, {\em \ 'El\'ements de g\'eom\'etrie alg\'ebrique: I. Le langage des sch\'emas}, Publications Math\'ematiques de l'IH\'ES 4, 1960, 5--228. 


\bibitem[Haz]{Haz} M. Hazewinkel, {\em Local class field theory is easy}, Adv. Math. {\bf 18}, 1975, 148--181.


\bibitem[KP]{KP} J.\ Kl\"uners, S.\ Pauli, {\em Computing residue class rings and Picard groups of orders}, J. Algebra {\bf 292}, 2005, 47--64. 

\bibitem[Knus]{Knus} M.\ A.\ Knus, {\em Quadratic and Hermitian Forms over Rings}, Grundlehren der mat. Wissenschaften {\bf 294}, 1991, Springer.




\bibitem[Mat]{Mat} H.\ Matsumura, {\em Commutative Ring Theory}, 2nd ed. Cambridge Studies in Advanced Mathematics 8. Cambridge Univ. Press, Cambridge, 1989.



\bibitem[Mor]{Mor} M.\ Morishita, {\em On S-class number relations of algebraic tori in Galois extensions of global fields}, Nagoya Math.\ J. {\bf 124}, 1991, 133--144.

\bibitem[Nis]{Nis} Y.\ Nisnevich, {\em \'Etale Cohomology and Arithmetic of Semisimple Groups}, Ph.D.~thesis, Harvard University, 1982.



\bibitem[Ono]{Ono} T.\ Ono, {\em On some class number relations for Galois extensions}, Nagoya Math J.~{\bf 107}, 1987, 121--133. 








\bibitem[Sko]{Sko} A.\ N.\ Skorobogatov, {\em Torsors and Rational Points}, Cambridge Tracts in Mathematics {\bf 144}, Cambridge Univ.~Press, 2001. 

\bibitem[TO]{TO} J.\ Tate, F.\ Oort, {\em Group schemes of prime order}, Ann. Sci. Ecole Norm. Sup. {\bf{3}}, 1970, 1--21.




\bibitem[Wat]{Wat} W.\ C.\ Waterhouse, {\em  Composition of norm-type forms}, J.~Reine Angew.~Math. {\bf 353}, 1984, 85--97. 



\end{thebibliography}
\end{document}